\documentclass{amsart}
\title[Projective Uniqueness of Polytopes]{An Algebraic Approach to Projective Uniqueness with an Application to Order Polytopes}
\author[Tristram Bogart, Jo\~{a}o Gouveia, and Juan Camilo Torres]{Tristram Bogart$^1$, Jo\~{a}o Gouveia$^2$, and Juan Camilo Torres$^1$}
\date{}
\address{$^1$ Departamento de Matem\'{a}ticas, Universidad de los Andes, Bogot\'{a}, Colombia}
\address{$^2$ University of Coimbra, CMUC, Department of Mathematics, Portugal}
\thanks{Emails: \MakeLowercase{tc.bogart22}@uniandes.edu.co, jgouveia@mat.uc.pt, jc.torresc@uniandes.edu.co}
\thanks{The first and third authors were supported by internal research grants (INV-2017-51-1453 and INV-2018-48-1373, respectively) from the Faculty of Sciences of the Universidad de los Andes. These grants allowed them to visit the second author and complete key steps of this project. The third author is also being supported in his doctoral studies, of which this project forms a part, by the Colombian science agency Colciencias. The second author was supported by the Centre for Mathematics of the University of Coimbra – UID/MAT/00324/2019, funded by the Portuguese Government through FCT/MEC and co-funded by the European Regional Development Fund through the Partnership Agreement PT2020.}
\usepackage{amsmath, amsthm}
\usepackage[foot]{amsaddr}
\usepackage{amscd,amssymb,amsfonts}
\usepackage[mathscr]{euscript}
\usepackage{enumerate}
\usepackage{stackengine}
\usepackage{tikz}
\usetikzlibrary{shapes}
\usetikzlibrary{decorations.pathreplacing}
\usetikzlibrary{arrows}
\usepackage{kbordermatrix}
\usepackage{dsfont}
\DeclareMathOperator{\rank}{rank}

\DeclareMathOperator{\row}{row}

\DeclareMathOperator{\conv}{conv}

\DeclareMathOperator{\Minors}{Minors}
\DeclareMathOperator{\Row}{Row}

\DeclareMathOperator{\K}{\\K}

\DeclareMathOperator{\Ord}{\mathcal{O}}
\DeclareMathOperator{\Pscr}{\mathscr{P}}
\DeclareMathOperator{\Qscr}{\mathscr{Q}}
\DeclareMathOperator{\Rscr}{\mathscr{R}}
\DeclareMathOperator{\CC}{\mathbb{C}}
\DeclareMathOperator{\G}{\mathbf{G}}

\DeclareMathOperator{\RR}{\mathbb{R}}

\DeclareMathOperator{\uu}{\mathbf{u}}
\DeclareMathOperator{\mm}{\mathbf{m}}
\DeclareMathOperator{\nn}{\mathbf{n}}
\DeclareMathOperator{\vv}{\mathbf{v}}
\DeclareMathOperator{\w}{\mathbf{w}}
\DeclareMathOperator{\x}{\mathbf{x}}
\DeclareMathOperator{\y}{\mathbf{y}}

\DeclareMathOperator{\p}{\mathbf{p}}

\DeclareMathOperator{\rev}{rev}
\DeclareMathOperator{\aff}{aff}
\DeclareMathOperator{\ver}{Vert}
\DeclareMathOperator{\facets}{Facets}

\newcommand{\cover}{\mathrel{\ooalign{$\prec$\cr\hidewidth\hbox{$\cdot$}\cr}}}

\theoremstyle{plain}
	\newtheorem{theorem}{Theorem}[section]
	\newtheorem{lemma}[theorem]{Lemma}
	\newtheorem{proposition}[theorem]{Proposition}
	\newtheorem{corollary}[theorem]{Corollary}
	\newtheorem{conjecture}[theorem]{Conjecture}

        \theoremstyle{definition}
	\newtheorem{definition}[theorem]{Definition}
	\newtheorem{example}[theorem]{Example}
	\newtheorem*{observation}{Observation}
	\newtheorem*{observations}{Observations}
\begin{document}

\begin{abstract} 
 A combinatorial polytope $P$ is said to be projectively unique if it has a single realization up to projective transformations. Projective uniqueness is a geometrically compelling property but is difficult to verify. In this paper, we merge two approaches to projective uniqueness in the literature. One is primarily geometric and is due to McMullen, who showed that certain natural operations on polytopes preserve projective uniqueness. The other is more algebraic and is due to Gouveia, Macchia, Thomas, and Wiebe. They use certain ideals associated to a polytope to verify a property called graphicality that implies projective uniqueness.

 In this paper, we show that that McMullen's operations preserve not only projective uniquness but also graphicality. As an application, we show that large families of order polytopes are graphic and thus projectively unique. 
\end{abstract}
\maketitle

\section{Introduction} \label{sec:intro}

A combinatorial polytope $P$ is said to be projectively unique if it has a single realization up to projective transformations. In other words, $P$ is projectively unique if any two (embedded) polytopes with its combinatorial structure can be mapped to each other by a projective transformation. Projectively unique polytopes form a very interesting class, where the combinatorics contain all essential information and the realization space is trivial in a very strong sense. In contrast, realization spaces of polytopes can be arbitrarily complicated in general \cite{R-G06}.

The study of projectively unique polytopes for their own sake goes back more than fifty years, with important pioneering work of Perles, Shepard and McMullen (see for instance \cite[Section 4.8]{Gru03},\cite{PS74} and \cite{Mc76}), but the full characterization of all three-dimensional projectively unique polytopes goes back even further, being a consequence of Steinitz's work in the early twentieth century \cite{S22}.  In dimension three, a polytope is projectively unique if and only if it has at most $9$ edges (simplices, square pyramids, triangular prisms and triangular bipyramids). In dimension two, it is a simple exercise to see that only triangles and quadrilaterals are projectively unique.

In higher dimensions, projective uniqueness is much more elusive. In dimension four, there is a list of $11$ combinatorial classes of projectively unique  polytopes conjectured to be complete by Shephard and McMullen (\cite{Mc76}). A weaker, more general question,  posed by Perles and Shephard in \cite{PS74} asks if the number of such combinatorial classes of polytopes in a fixed dimension $d \geq 4$ is even finite. This was answered negatively for $d \geq 96$ in \cite{AZ15}, but remains open for $4\leq d \leq 68$.

Our work merges two ideas from the literature on projective uniqueness. Very recently, the concept of \emph{slack ideals} introduced in \cite{GPRT17,GMTW18-1,GMTW18-2} presents an algebraic take on the study of realization spaces. In \cite{GMTW18-2} a subclass of projectively unique polytopes was defined, the \emph{graphic} polytopes, for which one has an algebraic certificate of projective uniqueness. On the other hand, an important early result in the study of projective uniqueness is due to  \cite{Mc76}, where certain operations on polytopes are introduced and proven to preserve projective uniqueness. Our main result is that these same operations also preserve graphicality. Note that this result neither implies nor is implied by the original McMullen result. This connection gives us an algebraic version of the McMullen's result that can be used to create large families of graphic polytopes, as well as to prove graphicality for particular polytopes of interest. 

Graphic polytopes are not only a subclass of projectively unique polytopes but also a subclass of \emph{morally 2-level} polytopes. Morally 2-level polytopes are those that have $0/1$ generalized slack matrices, and in particular they include all 2-level polytopes.  
Such polytopes play an important role in the theory of semidefinite representations of polyhedra and have been the focus of recent interest. Moreover, they comprise a very large family that includes many interesting polytopes; see for example \cite{ACF18} for combinatorially relevant examples and \cite{BFFFMP19} for a full enumeration in dimension up to $7$.

A perfect candidate to apply these methods is therefore the family of \emph{order polytopes}. Order polytopes were introduced in \cite{S86}. They are constructed from finite posets, and we can translate properties of the poset into properties of its order polytope. They are very interesting objects; their vertices and facets are easy to describe, and furthermore they are 0/1-polytopes and 2-level. Order polytopes offer us a combinatorial window on the phenomena we are studying, since we can understand visually at the level of posets the operations that are being applied to high-dimensonal polytopes, which are much harder to internalize. We use our main result to, in particular, prove that order polytopes from finite ranked posets with no 3-antichain are graphic, and therefore projectively unique. As a side effect, we obtain a tool to generate many low-dimensional, easy to understand, examples of graphic polytopes.

\textbf{Organization.} After this introductory section, section 2 contains background on projective uniqueness, slack ideals and order polytopes, providing a short review of the literature and stating the main results we will use. In section 3, we prove our main result, an algebraic analogue of McMullen's main result from \cite{Mc76}, stating that joins, vertex sums and (some) vertex splits of graphic polytopes are graphic. In section 4, we explore how various natural operations on posets are reflected in their order polytopes, and can be interpreted in terms of McMullen's operations. Finally, in section 5, we combine the results of Sections 3 and 4 to prove that if a finite ranked poset has no antichain of size 3, then its order polytope is graphic, and therefore projectively unique. We also propose some open problems related to graphicality and projective uniqueness of order polytopes.

\textbf{Notation.} As usual in combinatorics, for a positive integer $n$, $[n]:=\lbrace 1,\ldots,n\rbrace$. The sets of nonnegative and positive real numbers are denoted as $\mathbb{R}_{+}$ and $\mathbb{R}_{++}$, respectively. The affine span of a set $S \subseteq \RR^k$ is denoted by $\aff(S)$. Points in $\mathbb{R}^k$ are thought as column vectors, where $\mathbf{0}$ and $\mathbf{1}$ are, respectively, the all-zeros and the all-ones column vectors. For the all-zeros and all-ones row vectors we use the notations $\mathbf{0}^T$ and $\mathbf{1}^T$, respectively. The sizes of these vectors will be clear from the context. The set of real matrices of size $m\times n$ is denoted by $\mathbb{R}^{m\times n}$, and the zero matrix is denoted by $O$. A positive diagonal matrix is a real diagonal matrix with positive entries in the diagonal. Finally, we will use $\preceq$ for the order relation of an arbitrary finite poset and the symbol $\cover$ for the cover relation.

\section{Background} \label{sec:background}
\subsection{Projective Uniqueness and the McMullen operations} \label{sec:mcmullen}
When studying polytopes we usually do not want to consider a specific geometric realization of a polytope, but instead some equivalence class that preserves the properties we are interested in.
We begin by recalling three different types of equivalences between polytopes.
\begin{definition} 
Let $P, Q\subseteq\mathbb{R}^d$ be two full-dimensional polytopes.
\begin{enumerate}[a)]
\item We say that $P$ and $Q$ are \emph{combinatorially equivalent} if their face lattices are isomorphic as posets.
\item We say that $P$ and $Q$ are \emph{projectively equivalent} if there is a projective transformation $\phi:\mathbb{R}^d\dashrightarrow\mathbb{R}^d$ such that $\phi(P)=Q$.
\item We say that $P$ and $Q$ are \emph{affinely equivalent} if there is an affine transformation
    $\psi:\mathbb{R}^d\rightarrow\mathbb{R}^d$ such that $\psi(P)=Q$.

\end{enumerate}
\end{definition}

Recall that a projective transformation $\phi: \mathbb{R}^d\dashrightarrow\mathbb{R}^d$ is defined as
$$\phi(t)=\frac{A\mathbf{t}+\textbf{b}}{\textbf{c}^T\mathbf{t}+\gamma}$$
where $A\in\mathbb{R}^{d\times d}, \mathbf{b}, \mathbf{c}\in\mathbb{R}^d$, and $\gamma\in\mathbb{R}$ with
\[\det\begin{bmatrix}
A & \mathbf{b}\\
\mathbf{c}^T & \gamma
\end{bmatrix}\neq 0.\]

An affine transformation $\psi:\mathbb{R}^d\longrightarrow\mathbb{R}^d$ is defined as $\psi(\mathbf{t})=A\mathbf{t}+\mathbf{b}$ where $A\in\mathbb{R}^{d\times d}$ and $\mathbf{b}\in\mathbb{R}^d$. If $Q=\psi(P)$, then $A$ must be invertible due to the full-dimensionality of $P$ and $Q$.

\begin{observation}
For $P, Q \subseteq \RR^d$ full-dimensional polytopes, 
\[\text{affine equivalence} \Rightarrow \text{projective equivalence} \Rightarrow \text{combinatorial equivalence}.\]
\end{observation}

When studying geometric realizations of polytopes, projective transformations are the largest canonical class of maps from $\RR^d$ to $\RR^d$ that preserve the combinatorics of a polytope, so it is natural to
 consider realizations of polytopes up to projective equivalence. Occasionally, there is only one such realization.

\begin{definition} 
We say that a full-dimensional polytope $P$ is \emph{projectively unique} if any full-dimensional polytope that is combinatorially equivalent to $P$ is also projectively equivalent to $P$.
\end{definition}

\begin{example} The case of polygons illustrates the distinction between these three notions of equivalence. Any pair of triangles (or more generally, $d$-simplices) are affinely and thus projectively equivalent. Since affine transformations preserve parallel lines, a quadrilateral $Q$ is affinely equivalent to the square if and only $Q$ is a parallelogram. However, all quadrilaterals are projectively equivalent; that is, the square is projectively unique. For $m \geq 5$, the $m$-gon is not even projectively unique.
\end{example}

One of the biggest problems in the study of projectively unique polytopes is that we have few ways of constructing new examples. One of the most well-known ways, and the one we will apply in this paper, is to use certain operations proposed by McMullen \cite{Mc76} that preserve projective uniqueness. 

 The simplest operation we will consider is taking the dual of a polytope.
  \begin{definition}
Two polytopes $P$ and $P^*$ are \emph{duals} of each other if their face lattices are antisomorphic, that is, if there is an order-reversing bijection between these lattices.
  \end{definition}
  Note that a polytope $P \subseteq \RR^d$ containing the origin in its interior and its polar $P^\circ:=\{\x : \x^T\y \leq 1, \textrm{ for all } \y \in P\}$ are dual to each other. It can be shown that the dual of a projectively unique polytope is projectively unique.

  Apart from the dual, McMullen considers three additional constructions.
\begin{definition} \label{j}
  Let $P$ and $Q$ be polytopes of respective dimensions $d$ and $e$.
  \begin{enumerate}
  \item  Let $\hat{P}$ and $\hat{Q}$ be embeddings of $P$ and $Q$ in $\RR^{d+e+1}$ with nonintersecting affine spans whose underlying linear spaces intersect trivially. We define the \emph{join} of $P$ and $Q$ to be $P\vee Q:=\conv(\hat{P}\cup\hat{Q})$.
    
  \item Let $\vv$ and $\w$ be vertices of $P$ and $Q$, respectively. Let $\hat{P}$ and $\hat{Q}$ be embeddings of $P$ and $Q$ in $\RR^{d+e}$ whose affine spans intersect in a single point $\p$ which is the image of both $\vv$ and $\w$. Then  $P\oplus_{(\vv, \w)} Q:=\conv(\hat{P}\cup \hat{Q})$ is called the \emph{vertex sum} of $P$ and $Q$ along the pair $(\vv,\w)$. We will also write $P \oplus_{\p} Q := P \oplus_{(\vv, \w)} Q$. 


  \item Let $\p$ be a vertex of $P$.
    The polytope
\[ P_{\p}:=\conv(\lbrace (\mathbf{w},0):\mathbf{w}\in\ver(P),\mathbf{w}\neq\mathbf{p}\rbrace\cup\lbrace (\mathbf{p},1),(\mathbf{p},-1)\rbrace)\]
is called the \emph{vertex split} of $P$ along $\mathbf{p}$.

\end{enumerate}
\end{definition}

Note that we are sometimes identifying the constructions with specific embeddings for brevity of exposition, although we are interested in the combinatorial equivalence classes of these constructions. The combinatorial structure of all of these constructions is well known.

\begin{observation}
Let $P$ and $Q$ be polytopes with dimensions $d$ and $e$ (appropriately embedded, depending on the operation), vertex sets $V$ and $W$, and facet sets $\mathcal{F}$ and $\mathcal{G}$ respectively.  We then have the following structure.

\begin{tabular}{c|c|c|c}
Polytope   & dimension & vertex set    & facet set \\ \hline
$P \vee Q$ & $d+e+1$ & $V \cup W $ & $ \begin{array}{c} \{F \vee Q : F \in \mathcal{F}\}\\ \cup \\ \{P \vee G : G \in \mathcal{G}\} \end{array}$ \\ \hline
$P\oplus_{\p} Q$ & $d+e$ & $ \begin{array}{c}(V \setminus \p) \\ \cup \\ (W \setminus \p) \\ \cup \\ \{\p\} \end{array}$ &
$ \begin{array}{c} \{F \oplus_{\p} Q : \p \in F \in \mathcal{F}\}\\ \cup \\ \{P \oplus_{\p} G : \p \in G \in \mathcal{G}\} \\ \cup \\ \{F \vee G : \p \not \in F \in \mathcal{F} , \p \not \in G\in \mathcal{G}  \} \end{array}$ \\ \hline
$P_{\p}$ & $d+1$  & $(V \setminus \p) \cup \{\widehat{\p},\overline{\p}\}$ &
$\begin{array}{c} \{\conv(F \cup \{\widehat{\p},\overline{\p}\}) : \p \in F \in \mathcal{F}\} \\ \cup \\
                 \{ \conv(F \cup \{\widehat{\p}\}) : \p \not \in F \in \mathcal{F}\} \\ \cup \\  \{ \conv(F \cup \{\overline{\p}\}) : \p \not \in F \in \mathcal{F}\} \end{array}$\\
\end{tabular}

Note that $\widehat{\p}$ and $\overline{\p}$ denote $(\p,-1)$ and $(\p,1)$.
\end{observation}

%
%
%

In \cite{Mc76}, McMullen shows that these three operations preserve projective uniqueness, under certain mild conditions for the vertex splitting operation.

\begin{theorem}[\cite{Mc76}] \label{McOperations}
Let $P\subseteq\mathbb{R}^d$ and $Q\subseteq\mathbb{R}^e$ be two full-dimensional polytopes. Then:
\begin{enumerate}
\item $P \vee Q$ is projectively unique if and only $P$ and $Q$ are projectively unique, 
\item if $P$ and $Q$ are projectively unique, then so is $P\oplus_{(\vv,\w)} Q$ for any vertices $\vv$ of $P$ and $\w$ of $Q$, and 
\item  if $P$ is projectively unique and is not the vertex sum of two polytopes at $\p$, then $P_{\p}$ is projectively unique. 
\end{enumerate}
\end{theorem}


These operations are enough to construct from direct sums of simplices (easily shown to be projectively unique) all the $11$ known projectively unique $4$-polytopes, but necessarily cannot produce all projectively unique polytopes, as they generate only a finite list of $d$-dimensional examples for every fixed $d$. In fact they are not enough to generate even all projectively unique $5$-polytopes from the lower dimensional ones \cite[Theorem 4.4.1]{Wiebe}. Nevertheless, they are a very useful tool to construct new examples from existing ones. 

For some of the objects we will be studying, it will be useful to introduce the dual operations to vertex splitting and vertex sum that we will call, respectively, \emph{facet wedging} and \emph{facet product}. We will start by defining the facet wedge.
\begin{definition} \label{fw}
  Let $P\subseteq\mathbb{R}^d$ be a full-dimensional polytope and $F$ be a facet of $P$.
Then the polytope
\begin{align*}
P_F :=&  \conv(\lbrace (\mathbf{v},0) :{\mathbf{v}}\in \ver(P)\setminus \ver(F) \rbrace\cup\lbrace(\mathbf{v},1):{\mathbf{v}}\in \ver(P)\setminus \ver(F)\rbrace\\
 &\cup\lbrace(\mathbf{w},0) :\mathbf{w}\in \ver(F)\rbrace)
 \end{align*}
is called the \emph{facet wedge} of $P$ along $F$. 
\end{definition}
Note that if $P$ and $P^\ast$ are dual, then the vertex split of $P$ at $\p$ is dual to the facet wedge of $P^\ast$ along the facet $F$ that is dual to $\p$. Thus we can translate Theorem \ref{McOperations} (3) into a result about facet wedges.

We now turn our attention to facet products. This is a special case of a \emph{subdirect product}, introduced in \cite{Mc76} as a dual to the \emph{subdirect sum}, an operation that generalizes vertex sum.

\begin{definition} \label{fp}
  Let $P\subseteq\mathbb{R}^d$ and $Q\subseteq\mathbb{R}^e$ be two full-dimensional polytopes and $\hat{F}$ and $\bar{F}$ facets of $P$ and $Q$, respectively. We define the facet product of $P$ and $Q$ with respect to $\hat{F}$ and $\bar{F}$,  which we denote by $P \otimes_F Q$, as $(P^\ast \oplus_{\p} Q^\ast)^\ast$ where the vertex sum is with respect to the vertices of $P^\ast$ and $Q^\ast$ that are dual to $\hat{F}$ and $\bar{F}$, respectively.
\end{definition}

In practical terms, a precise geometric description of this polytope is not needed, as we are mostly concerned with its combinatorial structure.

\subsection{Slack Ideals and Graphic Polytopes}
Gouveia, Pashkovich, Robinson, and Thomas \cite{GPRT17} introduced the notion of the slack ideal of a polytope in order to study its positive semidefinite lifts. The first and last of these authors, along with Macchia and Wiebe \cite{GMTW18-1, GMTW18-2} then applied slack ideals to give an algebraic criterion for projective uniqueness. We now review the key definitions and results from these papers that form the starting point of our own work.

\begin{definition} 
  Let $P\subseteq\mathbb{R}^d$ be a full-dimensional polytope with vertices $\mathbf{v}_1,\ldots,\mathbf{v}_n$ and facets $F_1,\ldots,F_m$. Then $P=\lbrace \mathbf{t}\in\mathbb{R}^d: A\mathbf{t}+\mathbf{b}\geq\mathbf{0}\rbrace$ for some $A=[a_{ij}]_{m\times d}\in\mathbb{R}^{m\times d}$, $\mathbf{b}\in\mathbb{R}^m$, and such that $F_i=\lbrace \mathbf{t}\in P:a_{i1}t_1+\cdots+a_{id}t_d +b_i=0\rbrace$ for all $i=1,\ldots,m$.

If $h_i(t):=a_{i1}t_1+\cdots+a_{id}t_d +b_i$ for $i=1,\ldots,m$, then the matrix defined as $[h_i(\mathbf{v}_j)]_{m\times n}$ is called a \emph{slack matrix} of $P$. If we take a slack matrix of $P$ and replace each non-zero entry with a distinct variable, we obtain the \emph{symbolic slack matrix} of $P$.
\end{definition}
\begin{observations}
\text{}
\begin{itemize}
\item Slack matrices and the symbolic slack matrix of a polytope depend on the ordering of the vertices and the facets. So whenever we talk about these matrices, we will implicitly fix orderings on the vertices and the facets.
\item If we scale $h_i$ by a positive real number, we still obtain the same embedded polytope $P$. Thus if $S$ is a given slack matrix of $P$, then so is $DS$, where $D\in\mathbb{R}^{m\times m}$ is a positive diagonal matrix. Furthermore, all of the slack matrices of $P$ are of this form.
\end{itemize}
\end{observations}
It is not hard to see that any slack matrix of a $d$-dimensional polytope has rank $d+1$. It is also not hard to check that this is the minimum rank of any matrix with the same support.
\begin{lemma}[\cite{GMTW18-1}] \label{lr}
If  $S_P(x_1,\ldots,x_k)$ is the  symbolic slack matrix of a $d$-polytope $P$ and $\boldsymbol{\alpha}\in\mathbb{R}_{++}^k$, then $\rank S_P(\boldsymbol{\alpha})\geq d+1$.
\end{lemma}
In fact the rank plays a very important role in characterizing slack matrices. Given a polytope $P$, we call any slack matrix of a polytope combinatorially equivalent to $P$ a \emph{true slack matrix} of $P$. Any matrix that can be obtained from a true slack matrix by scaling columns by positive scalars is called a \emph{generalized slack matrix} of $P$.
\begin{theorem}[\cite{GGKPRT13}, \cite{GMTW18-1}]
Let $P$ be a $d$-polytope with symbolic slack matrix $S_P(x_1,\ldots,x_k)$, and $\boldsymbol{\alpha}\in\mathbb{R}_{++}^k$. Then
\begin{enumerate}[a)]
\item $S_P(\boldsymbol{\alpha})$ is a generalized slack matrix of $P$ if and only if $\rank S_P(\boldsymbol{\alpha})=d+1$
\item $S_P(\boldsymbol{\alpha})$ is a true slack matrix of $P$ if and only if $\rank S_P(\boldsymbol{\alpha})=d+1$ and $\mathbf{1}^T$ belongs to the row span of $S_P(\boldsymbol{\alpha})$.
\end{enumerate}
\end{theorem}
Generalized slack matrices have a more natural description than true slack matrices. Moreover, scaling rows and columns is a natural thing to do when studying polytopes up to projective equivalence. Two polytopes $P$ and $Q$ are projectively equivalent if $S_P=DS_QD'$, where $S_P$ and $S_Q$ are slack matrices of $P$ and $Q$ and $D$ and $D'$ are positive diagonal matrices. We can use this to give a characterization for projective uniqueness.
\begin{theorem}[\cite{GMTW18-1}]\label{tpu}
Let $P$ be a full-dimensional polytope. Then $P$ is projectively unique if and only if $P$ has only one generalized slack matrix up to column and row scaling by positive scalars.
\end{theorem}
We can extract from this geometric picture an algebraic version. To do that, for any matrix $M(\x)$ of constants and variables and any natural number $e$, denote the determinantal ideal of $e$-minors of $M(\x)$ by $\Minors_e\left(M(\x)\right)$. The condition on the rank of a slack matrix suggests consideration of the following ideal.
\begin{definition}[\cite{GPRT17}]
Let $P$ be a $d$-polytope with symbolic slack matrix $S_P(\mathbf{x})=S_P(x_1,\ldots,x_k)$. We define the \emph{slack ideal} of $P$ as
\[I_P:=\Minors_{d+2}\left(S_P(\mathbf{x})\right) : \left(\prod\limits_{i=1}^k x_i\right)^\infty.\]
\end{definition}

\begin{observation}
Saturation by the product of all variables removes common factors from the terms of a polynomial. More precisely, $\mm f\in I$ for some monomial $\mm$ if and only if $f \in I:\left( x_1 x_2 \cdots x_k\right)^\infty$. A small observation is that this also implies that if $\mm f\in I:\left( x_1 x_2 \cdots x_k\right)^\infty$ for some monomial  $\mm$, then
$f\in I:\left( x_1 x_2 \cdots x_k\right)^\infty$. This property will be used several times later on.

Geometrically, saturating an ideal $I$ has the effect of removing components of the variety $V(I)$ that are contained entirely in a coordinate hyperplane $x_i = 0$. This is a sensible operation for slack varieties because each variable represents the distance from a vertex $\vv$ to a facet that does not contain $\vv$.
\end{observation}

In \cite{GMTW18-2}, a second ideal associated to a polytope, $T_P$, is introduced in order to study projective uniqueness from an algebraic point of view. This new ideal is toric.
To introduce the ideal, we first define the \emph{non-incidence graph} of a polytope $P$. This is the bipartite graph $\mathbf{G}_P$ on $\facets(P)\sqcup\ver(P)$ with an edge connecting facet $F$ with vertex $\textbf{v}$ if and only if $\textbf{v}\notin F$. The edges of $\mathbf{G}_P$ are thus labeled by the variables that appear in the symbolic slack matrix of $P$.

To every collection $C$ of oriented edges in this graph we can associate a binomial in the following way: 
\begin{itemize}
\item let $x_1,\ldots,x_n$ be the labels (variables) of all edges of $C$ that, according to the orientation, go from $\facets(P)$ to $\ver(P)$
\item let $y_1,\ldots,y_m$ be the labels (variables) of all edges of $C$ that, according to the orientation, go from $\ver(P)$ to $\facets(P)$.
\end{itemize}
The binomial associated to $C$ is $f_C := x_1\cdots x_n-y_1\cdots y_m$. If $C$ is a simple cycle, we implicitly suppose its edges are oriented in order to form a directed cycle. So we can talk of the binomial $f_C$ associated to a simple cycle, which is unique up to sign.
\begin{definition}\label{tcc}
Let $P$ be a polytope. Then
\[ T_P  := \langle f_C: C \text{ is a chordless cycle of } \mathbf{G}_P\rangle.\]
\end{definition}
This is in fact the toric ideal associated to the vertex-edge incidence matrix of $\mathbf{G}_P$ (see \cite{GMTW18-2}). Also $T_P$ is generated by the binomials of all oriented cycles, a consequence of the Cycle-Splitting Lemma (Lemma \ref{lsc}) which will be introduced in the next section.

In general there is no obvious relationship between the ideals $I_P$ and $T_P$. However there is an important special case in which they are indeed related: that of $2$-level polytopes.
\begin{definition} A polytope $P$ is \emph{2-level} if for every facet $F$ of $P$, the vertices of $P$ that are not in $F$ are all contained in a single parallel translate of $\aff(F)$. Equivalently, $P$ is 2-level if $S_P(\mathbf{1})$ is a slack matrix of $P$. We will say that $P$ is \emph{morally 2-level} if $S_P(\mathbf{1})$ is a generalized slack matrix of $P$.
\end{definition}

\begin{theorem}[\cite{GMTW18-2}]\label{m2l}
Let $P$ be a full-dimensional polytope.
\begin{enumerate}
\item $P$ is morally 2-level if and only if $I_P\subseteq T_P$.
\item If $I_P=T_P$, then $P$ is projectively unique.
\end{enumerate}
\end{theorem}
\begin{definition}
  If $I_P=T_P$, we say that $P$ has a \emph{graphic slack ideal} and that $P$ is a \emph{graphic polytope}.
\end{definition}

Theorem \ref{m2l} implies that graphic polytopes are a subset of both morally $2$-level and projectively unique polytopes. Since in any given dimension the number of morally $2$-level polytopes must be finite (see \cite[\S 6]{ACF18}), so is the number of graphic polytopes, hence this must be a much more restrictive condition than simply being projectively unique. However, all the $11$ known examples of $4$-dimensional projectively unique polytopes are indeed graphic.

\subsection{Order Polytopes} \label{sec:orderpoly}
As mentioned above there are many known classes of 2-level polytopes for which one might be able to apply Theorem \ref{m2l}. We will focus on the following combinatorially appealing class of 2-level polytopes introduced by Stanley \cite{S86}.

\begin{definition}
Let $\Pscr=([d],\preceq)$ be a poset. The \emph{order polytope} of $\Pscr$ is
\[\Ord(\Pscr)=\lbrace\mathbf{t}\in\mathbb{R}^d:0\leq t_i\leq 1 \textnormal{ for all } i\in [d], \textnormal{ and } t_i\leq t_j \textnormal{ if } i\preceq j\rbrace.\]
\end{definition}

In our study of order polytopes we will use the following definitions related to posets.

\begin{definition} Let $\Pscr=([d],\preceq)$ be a poset.
  \begin{itemize}
  \item A subset of $[d]$ is a \emph{chain} (respectively \emph{antichain}) if its elements are pairwise comparable (respectively pairwise incomparable.)
  \item A subset $J$ of $[d]$ is called a \emph{filter} of $P$ if whenever $x\in J$ and $y\succeq x$, then $y\in J$. If $S\subseteq [d]$, the set $(S):=\lbrace x\in [d]: x\succeq s \textnormal{ for some } s\in S\rbrace$ is called the \emph{filter generated by $S$}.
  \item The function $\chi^{\Pscr}_S:=[d]\rightarrow\lbrace 0,1 \rbrace$ given by $f(x)=1$ if $x\in S$ and $f(x)=0$ if $x\notin S$ is called the \emph{characteristic function} of $S$. When the poset we are working with is fixed, we denote $\chi^{\Pscr}_S$ simply by $\chi_S$. We can also think of the characteristic function as a vector.
  \item The poset $\Pscr$ is \emph{ranked} if there is a function $\rho:[d]\rightarrow\mathbb{N}$ such that $\rho(y)=\rho(x)+1$ for all cover relations $x\cover y$ in $\Pscr$. (We do not insist that $\rho$ be unique.) The $\emph{rank}$ of a finite ranked poset is defined as the maximum length of a maximal chain.
  \end{itemize}
\end{definition}

There is a bijective correspondence between filters and antichains of $\Pscr$. Given a filter, take its minimal elements to obtain an antichain, and given an antichain, take the filter generated by this antichain.

\begin{theorem}[\cite{S86}]
Let $\Pscr=([d],\preceq)$ be a poset. Then $\Ord(\Pscr)\subseteq\mathbb{R}^d$ is a full-dimensional polytope. Its vertices are precisely the characteristic vectors $\chi_J$ where $J$ is a filter. Thus $\Ord(\Pscr)$ is a $0/1$-polytope. The facets are the following sets
\begin{itemize}
\item $\lbrace \mathbf{t}\in\Ord(\Pscr):t_i = 0\rbrace$ where $i$ is a minimal element of $\Pscr$,
\item $\lbrace \mathbf{t}\in\Ord(\Pscr):t_i =t_j\rbrace$ where $i\cover j$, and
\item $\lbrace \mathbf{t}\in\Ord(\Pscr):t_i = 1\rbrace$ where $i$ is a maximal element of $\Pscr$.
\end{itemize}
\end{theorem}
In other words, vertices are given by filters (or antichains) and the facets by covers, minimal and maximal elements of $\Pscr$. Each facet of $\Ord(\Pscr)$ is defined by an inequality of the form $t_i\geq0$, $t_i\leq t_j$, or $t_i\leq 1$ depending if the facet comes from a minimal element, a cover, or a maximal element. The notation $F:t_i\geq 0$ will mean that the facet $F$ comes from the minimal element $i$. Analogous notations will be used for the other types of facets.

\begin{example}  $ $
  \begin{enumerate}
  \item The empty poset has exactly one antichain: the empty set. Thus its order polytope is a single point.
  \item If $\Pscr$ is the chain $1\prec 2\prec\ldots\prec d-1\prec d$, then $\Ord(\Pscr)=\lbrace\mathbf{t}\in\mathbb{R}^d:0\leq t_1\leq t_2\leq\ldots\leq t_{d-1}\leq t_d\leq 1\rbrace$; that is, $\Ord(\Pscr)$ is a $d$-simplex.
  \item If $\Pscr$ is an antichain with $d$ elements, then $\Ord(\Pscr)=\lbrace\mathbf{t}\in\mathbb{R}^d:0\leq t_i\leq 1 \text{ for all }1\leq i\leq d\rbrace$; that is, $\Ord(\Pscr)$ is a $d$-cube.
  \end{enumerate}
\end{example}

  \section{Operations on Polytopes that Preserve Graphicality} \label{sec:operations}
  We saw in Section \ref{sec:mcmullen} the operations on (combinatorial) polytopes introduced in \cite{Mc76} by McMullen that preserve projective uniqueness, as seen in Theorem \ref{McOperations}. In this section we will show that these operations also preserve graphicality. Note that this neither implies nor is implied by the results of McMullen, since we have both a stronger hypothesis and a stronger conclusion. To be more precise, we will prove the following graphical version of  Theorem \ref{McOperations}.

  \begin{theorem} \label{thm:main}
    Let $P\subseteq\mathbb{R}^d$ and $Q\subseteq\mathbb{R}^e$ be two full-dimensional polytopes. Then:
\begin{enumerate}
\item $P \vee Q$ is graphic if and only if $P$ and $Q$ are graphic, 
\item if $P$ and $Q$ are graphic, then so is $P\oplus_{(\vv,\w)} Q$ for any vertices $\vv$ of $P$ and $\w$ of $Q$, and
\item if $P$ is graphic and is not the vertex sum of two polytopes at $\p$, then $P_{\p}$ is graphic.
\end{enumerate}
  \end{theorem}

This is a direct analogue of Theorem \ref{McOperations}.  Note that while the first two parts of Theorem \ref{thm:main} show that some operations unconditionally preserve graphicality, the same is not true for the last part. The condition of not being a vertex sum of two polytopes is not very natural in an algebraic setting. We therefore derive a necessary and sufficient algebraic condition (fully described later in the section) for a vertex split to be projectively unique. This condition can be easier to check than McMullen's original geometric condition.

 Before setting out to prove this theorem, we will first show that duality preserves graphicality.
\begin{proposition}\label{d}
Let $P$ and $P^*$ be dual polytopes. Then  $P$ is graphic if and only if $P^*$ is graphic.
\end{proposition}
\begin{proof}
Since the symbolic slack matrices of $P$ and $P^*$ are transposes of each other, we have $I_P=I_{P^*}$. Also, the non-incidence graphs of $P$ and $P^*$ are the same and thus $T_P=T_{P^*}$. From these equalities, the result follows.
\end{proof}
Proposition \ref{d} allows us to translate results from vertex splitting to facet wedging, which will be convenient later on in the context of order polytopes.
\subsection{Auxiliary results}
To show that the McMullen operations also preserve graphicality, we will analyze their effect on the slack ideals $I_P$ and the toric ideals $T_P$. To do that we will use two technical auxiliary results that we present in this subsection.

The first of these results is simply a restatement of the usual argument used to show that slack matrices of $d$-dimensional polytopes have rank $d+1$
by showing that certain submatrices associated to flags of faces are triangular.

  \begin{lemma}[Flag Lemma \cite{GMTW18-1}] \label{lem:flag} Let $P$ be a $d$-polytope and
    \[ \emptyset = g_{-1} \subset g_0 \subset \dots \subset g_{d-1} \subset g_d = P \]
    be a complete flag of faces of $P$. Let $G_0, \dots, G_d$ be facets of $P$ such that
  $g_k = G_d \cap G_{d-1} \cap \dots G_{k+1}$ for $k = -1,0,\dots,d-1$ and $\w_0, \dots, \w_d$ be vertices of $P$ such that $\w_k \in g_k \setminus g_{k-1}$ for $k=0,\dots,d$. Then the $(d+1) \times (d+1)$ submatrix $A(\x)$ formed from the rows of $S_P(\x)$ indexed by $G_0,\dots,G_d$ and the columns indexed by $\w_0,\dots,\w_d$ is upper triangular with variables on the diagonal. In particular, the $(d+1)$-minor given by its determinant is a nonzero monomial.
\end{lemma}
\begin{proof}
  For each $j = 0,\dots,d$, we have $\w_j \in g_j = G_d \cap \dots \cap G_{j+1}$, so $a_{ij} = 0$ for each $i > j$. But $\w_j \notin g_{j-1} = G_d \cap \dots \cap G_{j+1} \cap G_j$, so $\w_j \notin G_j$ and so $a_{jj}$ is a variable for each $j$. That is, $A(\x)$ is upper triangular with variables on the diagonal.
\end{proof}

The second basic result we will be repeatedly using relates the cycle space of the non-incidence graph of a polytope with its slack ideal.

Here, if $C$ is a collection of oriented edges of $\G_P$, then $\overline{C}$ denotes the same collection of edges but with opposite orientations. Also, if $C_1$ and $C_2$ are collections of oriented edges, then we can find collections $C'_1$, $C'_2$ and $C_0$ such that $C_1=C'_1\cup C_0$, $C_2=C'_2\cup\overline{C_0}$, and there is no edge that is in both $C'_1$ and $C'_2$ but with different orientation. Using this, we define $C_1+C_2:=C'_1\cup C'_2$.

\begin{lemma}[Cycle-Splitting Lemma]\label{lsc}
  Let $C_1$ and $C_2$ be collections of oriented edges of $\mathbf{G}_P$.
  If $f_{C_1}$ and $f_{C_2}$ both belong to $I_P$, then $f_{C_1+C_2} \in I_P$.
\end{lemma}
\begin{proof}
Using the notation above, let $f_{C'_1}:=\mm_1 - \nn_1$, $f_{C'_2}:=\mm_2 - \nn_2$, and $f_{C_0}=\mm_3 - \nn_3$. Thus $f_{C_1}=\mm_1\mm_3 - \nn_1\nn_3$ and $f_{C_2}:=\mm_2\nn_3 - \nn_2\mm_3$. Since $f_{C_1}$ and $f_{C_2}$ are in $I_P$, so is $\mm_2 f_{C_1} + \nn_1 f_{C_2} = \mm_1 \mm_2 \mm_3 - \nn_1 \nn_2 \mm_3 = \mm_3 f_{C_1+C_2}$. Since $I_P$ is saturated, we conclude that $f_{C_1+C_2} \in I_P$.
%
%
%
%
\end{proof}
We call this result the Cycle-Splitting Lemma because we are going to apply it in the case that $C_1$, $C_2$ and $C_1+C_2$ are cycles.

\subsection{The join operation}
The first and simplest of the operations introduced by McMullen is that of the join. An important special case of this construction is the pyramid over a polytope $P$, which is the join of $P$ with a point.
The slack ideal of the pyramid is exactly the same as the slack ideal of the original polytope, so all algebraic properties, including graphicality, are preserved.

Even in the general case, the join operation is very easy to interpret in terms of slack matrices. In fact, from the description of its facets and vertices in Section \ref{sec:mcmullen} we see that the symbolic slack matrix of $P\vee Q$ is given by
\[S_{P\vee Q}(\x,\y)=
\begin{bmatrix}
S_P(\x) & O\\
O & S_Q(\y)
\end{bmatrix}\]
where $S_P(\x)$ and $S_Q(\y)$ are the symbolic slack matrices of $P$ and $Q$, respectively. This makes the slack ideal of the join easy to describe in terms of the original slack ideals.

\begin{lemma}\label{lem:joinslack}
Let $P\subseteq\mathbb{R}^d$ and $Q\subseteq\mathbb{R}^e$ be two full-dimensional polytopes and $I_P \subseteq \RR[\x]$ and $I_Q \subseteq \RR[\y]$ their slack ideals. Then $$I_{P\vee Q}=\langle I_P \rangle+\langle I_Q\rangle.$$
Moreover $I_{P\vee Q} \cap \RR[\x]=I_P$ and $I_{P\vee Q} \cap \RR[\y]=I_Q$.
\end{lemma}
\begin{proof}
  We first note that by \cite[Lemma 2.6]{SW19}, $\langle I_P \rangle+\langle I_Q\rangle$ is saturated because $I_P$ and $I_Q$ are saturated ideals in two polynomial rings on disjoint sets of variables. Thus, to prove $I_{P\vee Q}\subseteq \langle I_P \rangle+\langle I_Q\rangle$ it is enough to show that any $(d+e+3)$-minor of $S_{P \vee Q}$ is in $\langle I_P \rangle + \langle I_Q \rangle $. Let $m$ be a nonzero $(d+e+3)$-minor of $S_{P \vee Q}(\x,\y)$. By the block structure of $S_{P \vee Q}$, $m(\x,\y) = m'(\x)m''(\y)$ where $m'$ is an $r$-minor of $S_P(\x)$ and $m''$ is an $s$-minor of $S_P(\y)$ for some $r$ and $s$ that sum to $d+e+3$. By the pigeonhole principle, either $r \geq d+2$ or $s \geq e+2$. Without loss of generality, assume the former.  It is easy to see that $\Minors_{r}\left(S_P(\x)\right) \subseteq \Minors_{d+2}\left(S_P(\x)\right)$ by way of Laplace expansion, so $m'(\x) \in I_P$ and hence $m(\x,\y) \in \langle I_P \rangle$, concluding the proof of the forward inclusion.



To prove the reverse inclusion it is enough to show that any $(d+2)$-minor of $S_P(\x)$ is in $I_{P\vee Q}$. The symmetry between $P$ and $Q$ and the fact that $I_{P\vee Q}$ is saturated then will imply it. Let $m(\x)$ be the $(d+2)$-minor of $S_P(\x)$ obtained from an arbitrary $(d+2) \times (d+2)$ submatrix $T$. By the Flag Lemma, we know there is some $(e+1)\times(e+1)$ triangular submatrix $U$ of $S_Q(\y)$ with variables on the diagonal. Form a $(d+e+3)\times (d+e+3)$ submatrix of $S_{P \vee Q}(\x,\y)$ containing both $T$ and $U$. The block structure will imply that the associated $(d+e+3)$-minor is simply $m(\x)n(\y)$ where $n(\y)$ is a monomial in the variables $\y$. Thus $m(\x)n(\y) \in I_{P \vee Q}$, and by saturation this implies $m(\x)$ belongs to $I_{P\vee Q}$.

We are left to prove that $I_{P\vee Q} \cap \RR[\x]=I_P$ since the analogous result for $Q$ follows from the symmetry of the construction. We prove the forward inclusion since the other is clear. Let $f(\x) \in I_{P \vee Q} = \langle I_P \rangle + \langle I_Q \rangle$. Then there are polynomials $f_i \in I_P, g_j \in I_Q$ and $p_i, q_j \in \CC[\x,\y]$ such that
\[f(\x) = \sum_{i} p_i(\x, \y)f_i(\x) + \sum_{j} q_j(\x, \y)g_j(\y).\]
Evaluating the expression at some
$\tilde{\y}$ such that $S_Q(\tilde{\y})$ is a true slack matrix of $Q$, we have that $g_j(\tilde{\y})=0$ hence
\[ f(\x) = \sum_{i} p_i(\x, \tilde{\y}) f_i(\x) \in I_P,\]
proving the claim. 
\end{proof}

The join is even easier to understand when applied to $T_{P}$ and $T_Q$.
Since the non-incidence graph of $P\vee Q$ is just the disjoint union of the non-incidence graphs of $P$ and $Q$, we have by definition that
$$T_{P \vee Q} = \langle T_P \rangle + \langle T_Q \rangle.$$
By the same argument used above but with $\tilde{\y}$ being the all-ones vector, we obtain that $T_{P \vee Q} \cap \RR[\x]=T_P$ and
$T_{P \vee Q} \cap \RR[\y]=T_Q$. From these properties it is easy to show that the join preserves graphicality.

\begin{theorem}\label{tj}
Let $P\subseteq\mathbb{R}^d$ and $Q\subseteq\mathbb{R}^e$ be two full-dimensional polytopes. Then
\begin{enumerate}[a)]
\item $I_{P\vee Q}\subseteq T_{P\vee Q}$ if and only if $I_P\subseteq T_P$ and $I_Q\subseteq T_Q$.
\item $T_{P\vee Q}\subseteq I_{P\vee Q}$ if and only if $T_P\subseteq I_P$ and $T_Q\subseteq I_Q$.
\item In particular, $P\vee Q$ is graphic if and only if $P$ and $Q$ are graphic.
\end{enumerate}
\end{theorem}
\begin{proof}
  The first two statements follow from Lemma \ref{lem:joinslack} and the properties of $T_{P \vee Q}$ seen above by intersecting with $\RR[\x]$ and $\RR[\y]$. The third statement follows immediately from the first two.
\end{proof}


\begin{observation}
Note that by Theorem \ref{m2l} (1), a) is equivalent to saying that $P\vee Q$ is morally 2-level if and only if $P$ and $Q$ are morally 2-level. Also, from the block structure of $S_{P\vee Q}(\x,\y)$, it is clear that $P\vee Q$ is projectively unique if and only if $P$ and $Q$ are projectively unique (see Theorem \ref{tpu}.)
\end{observation}

\subsection{The vertex sum operation}

We now turn to a more involved operation: the vertex sum. Before
proving that this operation preserves graphicality, we will explain its effect on slack matrices. For simplicity we work with the \emph{support} of each symbolic slack matrix, which is the matrix obtained by replacing each variable by a 1.

Let $P$ be a $d$-dimensional polytope and $Q$ an $e$-dimensional polytope, embedded in such a way that their affine spans intersect only in a single common vertex $\p$ as in Definition \ref{j} (2). Denote by $\mathcal{F}$ and $\mathcal{G}$ the sets of facets of $P$ and $Q$, respectively, that contain $\p$, while $\overline{\mathcal{F}}$ and $\overline{\mathcal{G}}$ are the sets of facets that do not contain $\p$. Suppose that $\overline{\mathcal{F}}$ has $r$ elements while $\overline{\mathcal{G}}$ has $s$. Furthermore, let $V:=\ver(P)\setminus\lbrace \p\rbrace$ and $W:=\ver(Q)\setminus\lbrace \p\rbrace$. With this notation, the supports of the slack matrices of $P$ and $Q$ are as follows:
\[
S_P(\mathbf{1})=\kbordermatrix{
& V & \p\\
\mathcal{F} & A & \mathbf{0}\\
\overline{\mathcal{F}} & \overline{A} & \mathbf{1}_r
}\ \text{ and }\ S_Q(\mathbf{1})=\kbordermatrix{
& W & \p\\
\mathcal{G} & B & \mathbf{0}\\
\overline{\mathcal{G}} & \overline{B} & \mathbf{1}_s
}.
\]
From the description of the vertices and facets of a vertex sum in Section \ref{sec:mcmullen}, we can conclude that the support of the slack matrix of $P\oplus_{\p}Q$ has the form
\[
S_{P\oplus_{\p}Q}(\mathbf{1})=\kbordermatrix{
& V & \p & W\\
\mathcal{F}\oplus_{\p}Q & A & \mathbf{0} & O\\
P\oplus_{\p}\mathcal{G} & O & \mathbf{0} & B\\
\overline{\mathcal{F}} \vee \overline{\mathcal{G}} & \overline{A} \otimes \mathbf{1}_s & \mathbf{1}_{r+s} & \mathbf{1}_r \otimes \overline{B}
},\]
where $\mathcal{F}\oplus_{\p}Q:=\lbrace F\oplus_{\p} Q:F\in\mathcal{F}\rbrace$, $P\oplus_{\p}\mathcal{G}:=\lbrace P\oplus_{\p} G:G\in\mathcal{G}\rbrace$ and  $\overline{\mathcal{F}}\vee\overline{\mathcal{G}}:=\lbrace F\vee G:F\in\overline{\mathcal{F}},G\in\overline{\mathcal{G}}\rbrace$. Note that the rows indexed by the facets in $\overline{\mathcal{F}}\vee\overline{\mathcal{G}}$ are all possible concatenations of rows of $\overline{A}$ and rows of $\overline{B}$ with a 1 in the middle.

With this observation, it is easy to show that if $P$ and $Q$ are morally $2$-level then so is $P\oplus_{\p}Q$. Recall that $P$ and $Q$ being morally $2$-level just means that $S_P(\mathbf{1})$ and $S_P(\mathbf{1})$ are of rank $d+1$ and $e+1$, respectively.  Looking at  the submatrix of $S_{P\oplus_{\p}Q}(\mathbf{1})$ whose columns are indexed by $V$ and $\p$, we see that the rows there are precisely those of $S_P(\mathbf{1})$ with some repetitions, so it has rank $d+1$, while the submatrix of the columns indexed by $W$ and $\p$ similarly has rank $e+1$. Since the two submatrices share a column, the total rank of $S_{P\oplus_{\p}Q}(\mathbf{1})$ is at most $(d+1) + (e+1) -1 = d+e+1$. Since ${P\oplus_{\p}Q}$ is $(d+e)$-dimensional, the rank must in fact be exactly $d+e+1$, so the vertex sum is indeed morally $2$-level. Using Theorem \ref{m2l} to translate this result to an algebraic language, we obtain the following.

\begin{lemma}\label{lem:mor2levelvertexsum}
Let $P,Q$ be polytopes. If $I_P\subseteq T_P$ and $I_Q\subseteq T_Q$, then $I_{P\oplus_{\p} Q}\subseteq T_{P\oplus_{\p} Q}$.
\end{lemma}

It remains to show that the same implication holds when we reverse the inclusions. This requires more involved reasoning as we see next.

  \begin{theorem}\label{tvs}
  Let $P,Q$ be polytopes. If $T_P\subseteq I_P$ and $T_Q\subseteq I_Q$, then $T_{P\oplus_{\p} Q}\subseteq I_{P\oplus_{\p} Q}$. In particular, if $P$ and $Q$ are graphic then $P\oplus_{\p} Q$ is graphic.
  \end{theorem}
  
\begin{proof}
The second statement follows from the first one together with Lemma \ref{lem:mor2levelvertexsum}, so it is enough to prove the first.

 Suppose $T_P \subseteq I_P$ and $T_Q \subseteq I_Q$ and let $C$ be a chordless cycle of $\mathbf{G}_{P\oplus_{\p} Q}$. We have to show that the binomial associated to $C$ is in $I_{P\oplus Q}$.

 Again let $V = \ver(P) \setminus \{ \p \}$ and $W = \ver(Q) \setminus \{ \p \}$. We will continue the proof by cases depending on the form of $C$.

  \textbf{Case 1:} Suppose $C$ contains only vertices from one polytope, which we can by symmetry consider to be $P$. In other words suppose $C$ contains no vertex from $W$, and comes from the submatrix indexed by $V$ and $\p$.
  
  If $C$ does not contain any pair of facets $F \vee G_j$, $F \vee G_k$, for $j \neq k$, then it does not use any pair of rows with the same support (recall that the nonzero rows of this submatrix are just copies of the rows of $S_P$), and therefore there is a submatrix $S$ of $S_{P\oplus_{\p} Q}$ that indexes all the facets and vertices in $C$ and has the same support as $S_P$. Let $T$ be any  $(d+2) \times (d+2)$ submatrix of $S$.

  Note that the submatrix $B$ of $S_Q$ has an $e \times e$ triangular submatrix $M$. To see this, apply the Flag Lemma to it with a flag starting at the vertex $\p$, and then remove the column indexed by $\p$ and the row indexed by the unique facet that does not contain $\p$. Now $\det(T)\cdot\det(M)$ is a $(d+e+2)$-minor of $S_{P\oplus_{\p} Q}$ by the block structure of this matrix. Since $\det(M)$ is a monomial, this proves that the saturation of the ideal $I$ generated by the $(d+2)$-minors of $S$ is contained in $I_{P\oplus_{\p} Q}$. Since $I$ is isomorphic to $I_P$, it contains the binomial $f_C$ by the graphicality of $P$. Thus $f_C\in I\subseteq I_{P\oplus_{\p} Q}$ as we needed to show.  


  Now suppose that $C$ uses a pair of repeated rows; that is to say, it contains a pair of facets $F \vee G_1$, $F \vee G_2$ where $F$ is a facet of $P$ and $G_1, G_2$ are distinct facets of $Q$. Since they are connected to precisely the same nodes in $V \cup \{ \p \}$, then the chordlessness of $C$ forces it to be of the form
   \[C = \uu_1 (F \vee G_1) \uu_2 (F \vee G_2) \uu_1, \]
where $\uu_1,\uu_2\in V \cup \{ \p \}$, which yields the binomial $f_C = z_{11}z_{22}-z_{12}z_{21} \in T_{P\oplus_{\p}Q}$.

  Remember that the submatrix $B$ of $S_Q$ has an $e \times e$ triangular submatrix $M$. Furthermore, by applying the Flag Lemma to $P$ with a flag ending at $F$, $S_P$ has a $(d+1)\times(d+1)$ upper triangular submatrix $U$ where the bottom row is indexed by $F$ and the rightmost column by $\mathbf{u}_1$. Let $N$ be the matrix obtained from $U$ by removing that row and column. Looking at the structure of $S_{P\oplus_{\p} Q}$, we then see that it has a $(d+e+2)\times (d+e+2)$ submatrix of the form
\[
\kbordermatrix{
          &   &   & \mathbf{u}_1 & \mathbf{u}_2 \\
          & M & O & \mathbf{0}   & \mathbf{0}   \\
          & * & N & *   & *   \\
F\vee G_1 & * & \mathbf{0}^T & z_{11} & z_{12} \\
F\vee G_2 & * & \mathbf{0}^T & z_{21} & z_{22} }.\]
The determinant of this submatrix equals $\det(M) \cdot \det(N) \cdot f_C$ which is a monomial times $f_C$, and thus $f_C \in I_{P \oplus_{\p} Q}$.

  \textbf{Case 2:} Suppose now that $C$ contains a vertex of $V$ and a vertex of $W$. By analyzing the structure of the slack matrix, we see that to go from a vertex in $V$ to a vertex in $W$ in the graph $G_{P \oplus_{\p} Q}$, one must go through a facet of the form $F \vee G$. For the cycle $C$ to pass through $V$ and $W$, that must happen at least twice. Since $\p$ is connected to all those facets, if $\p$ were in $C$ it would have a chord. Hence we know that $\p \not \in C$ and so
 we can write
  \[ C = (F_0 \vee G_0) A_0 (F_1 \vee G_1) A_1 \dots A_{2k-2}(F_{2k-1} \vee G_{2k-1})A_{2k-1}(F_0 \vee G_0)\]
  where $A_i$ is a path whose vertices are all in $V$ (if $i$ is even) or all in $W$ (if $i$ is odd) and $F_i$ (resp. $G_i$) are facets of $P$ (resp. of $Q$) that do not contain $\p$. But then there is an edge in the nonincidence graph from $\p$ to $F_i \vee G_i$ for each $i$, so the graph contains cycles $C_0, \dots, C_{2k-1}$ where
  \[ C_i = (F_i \vee G_i)A_i(F_{i+1} \vee G_{i+1}) \p (F_i \vee G_i)\]
  with indices taken modulo $2k$. In particular, $C_i$ never contains both a vertex in $V$ and a vertex in $W$, so for each $i$, $f_{C_i} \in  I_{P\oplus_{\p}Q}$ by the previous case. Then by the Cycle-Splitting Lemma (Lemma \ref{lsc}), $f_C \in I_{P\oplus_{\p}Q}$.
\end{proof}

\subsection{The vertex splitting operation}

We proceed now with the operation of vertex splitting. The study of graphicality under this operation turns out to be more delicate, as it is not the case that it is unconditionally preserved.

In terms of slack matrices, the operation of vertex splitting is again quite simple. As in the previous section, denote by $\mathcal{F}$ and $\overline{\mathcal{F}}$ the sets of facets of $P$  that, respectively, contain or do not contain $\p$. Furthermore, let $V:=\ver(P)\setminus\lbrace \p\rbrace$. If the symbolic slack matrix of $P$ has support of the form
\[
S_P(\mathbf{1})=\kbordermatrix{
& V & \p\\
\mathcal{F} & A & \mathbf{0}\\
\overline{\mathcal{F}} & \overline{A} & \mathbf{1}
}\, \]
then the support of the slack matrix of the vertex split at $\p$ is
\[
S_{P_{\p}} = \kbordermatrix{
& V & \overline{\p} & \widehat{\p}\\
\mathcal{F} & A & \mathbf{0} & \mathbf{0}\\
\overline{\mathcal{F}}  & \overline{A} & \mathbf{1} & \mathbf{0}\\
\widehat{\mathcal{F}} & \overline{A} & \mathbf{0} & \mathbf{1}
}.\ \]
As with vertex sums, it is not difficult to see that this preserves moral $2$-levelness. We need to prove that $S_{P_{\p}}(\mathbf{1})$ has rank $d+2$ when $S_P(\mathbf{1})$ has rank $d+1$. Indeed, we get that the submatrix of $S_{P_{\p}}(\mathbf{1})$ with rows indexed by $\mathcal{F}$ and $\overline{\mathcal{F}}$ is exactly $S_P(\mathbf{1})$ with a zero column added, hence has also rank $d+1$. Adding the rows indexed by $\widehat{\mathcal{F}}$ increases the rank by just one, since they can all be attained from the corresponding rows of
$\overline{\mathcal{F}}$ by adding $\begin{bmatrix}0 & \cdots & 0 & -1 & 1\end{bmatrix}$ which is the difference between the first row indexed by $\widehat{\mathcal{F}}$ and the first row indexed by $\overline{\mathcal{F}}$. In algebraic terms, we have just proved the following result.

\begin{lemma}
 Let $P_{\p}$ be obtained from $P$ by splitting a vertex $\p$. If $I_P\subseteq T_P$, then $I_{P_{\p}}\subseteq T_{P_{\p}}$.
\end{lemma}

In order to study the other inclusions, we need to consider more closely the structure of the non-incidence graph of the vertex split. In Figure \ref{fig:graphvsplit} one can see the structure of the graph of the vertex split described above.

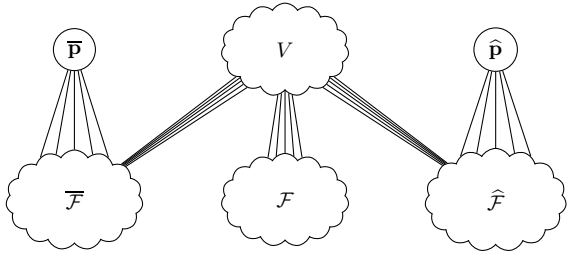
\begin{figure}[ht]
\begin{tikzpicture}[scale=0.8, every node/.style={scale=0.8}]

    \node[cloud, draw,cloud puffs=15,cloud puff arc=120, aspect=1.5, inner ysep=1em] (1) at (0,1.5) {$V$};
    \node[cloud, draw,cloud puffs=15,cloud puff arc=120, aspect=1.5, inner ysep=1em] (2) at (-3.5, -1)  {$\overline{\mathcal{F}}$};
    \node[cloud, draw,cloud puffs=15,cloud puff arc=120, aspect=1.5, inner ysep=1em,minimum width=0.5cm] (3) at (3.5, -1) {$\widehat{\mathcal{F}}$};
    \node[cloud, draw,cloud puffs=15,cloud puff arc=120, aspect=1.5, inner ysep=1em,minimum width=0.5cm] (4) at (0, -1) {$\mathcal{F}$};
    \node[draw, circle] (5) at (-3.5,1.5) {$\overline{\p}$};
    \node[draw, circle] (6) at (3.5,1.5) {$\widehat{\p}$};

\draw (1) -- (4);
\draw ([xshift=0.5cm]1) -- ([xshift=0.5cm]4);
\draw ([xshift=-0.5cm]1) -- ([xshift=-0.5cm]4);
\draw ([xshift=-1cm]1) -- ([xshift=-1cm]4);
\draw ([xshift=1cm]1) -- ([xshift=1cm]4);

\draw (1) -- (2);
\draw ([xshift=0.5cm]2) -- ([xshift=2cm]1);
\draw ([xshift=-0.5cm]2) -- ([xshift=-2cm]1);
\draw ([xshift=-1cm]2) -- ([xshift=-4cm]1);
\draw ([xshift=1cm]2) -- ([xshift=4cm]1);

\draw (1) -- (3);
\draw ([xshift=0.5cm]3) -- ([xshift=1cm]1);
\draw ([xshift=-0.5cm]3) -- ([xshift=-1cm]1);
\draw ([xshift=-1cm]3) -- ([xshift=-2cm]1);
\draw ([xshift=1cm]3) -- ([xshift=2cm]1);

\draw (6) -- (3);
\draw ([xshift=2cm]6) -- ([xshift=2cm]3);
\draw ([xshift=-2cm]6) -- ([xshift=-2cm]3);
\draw ([xshift=-1cm]6) -- ([xshift=-1cm]3);
\draw ([xshift=1cm]6) -- ([xshift=1cm]3);

\draw (5) -- (2);
\draw ([xshift=2cm]5) -- ([xshift=2cm]2);
\draw ([xshift=-2cm]5) -- ([xshift=-2cm]2);
\draw ([xshift=-1cm]5) -- ([xshift=-1cm]2);
\draw ([xshift=1cm]5) -- ([xshift=1cm]2);

\end{tikzpicture}
\caption{Structure of the non-incidence graph of a vertex split\label{fig:graphvsplit}}
\end{figure}

We will be especially interested in two classes of cycles. The first is the class of $4$-cycles of the form $$ \overline{F} \vv \widehat{F} \w \overline{F} $$ where $\vv,\w \in V$, and
$\overline{F} \in \overline{\mathcal{F}}$ and  $\widehat{F} \in \widehat{\mathcal{F}}$ are associated to the same facet $F$ of $P$. We will say these are cycles \emph{of type A}. The second class, the \emph{cycles of type B}, are those $8$-cycles of the form $$\overline{\p} \overline{F} \vv \widehat{F} \widehat{\p} \widehat{G} \w \overline{G} \overline{\p}$$
where, as before,  $\vv, \w \in V$,
$\overline{F}, \overline{G} \in \overline{\mathcal{F}}$ and  $\widehat{F}, \widehat{G} \in \widehat{\mathcal{F}}$,
$\overline{F}$ and $\widehat{F}$ are derived from the same facet $F$ of $P$ and similarly with $\overline{G}$ and $\widehat{G}$. It turns out that these are the two types of cycles that must be checked in order for the inclusion of the toric ideal in the slack ideal to be maintained. 

\begin{proposition} \label{prop:vertsplitcycleclasses}
Let $P_{\p}$ be the vertex split of $P$ as above. If $T_P \subseteq I_P$, then $T_{P_{\p}}\subseteq I_{P_{\p}}$ if and only if the binomials associated to all cycles of type A and B are in $I_{P_{\p}}$.
\end{proposition}
\begin{proof}
We simply need to show that if all the binomials associated to cycles of type A and B are in $I_{P_{\p}}$, then for any oriented cycle $C$ of $\G_{P_{\p}}$, its associated binomial $f_C$ is in $I_{P_{\p}}$, as the other implication is trivial. We will prove this by induction on the number of nodes of $C$ belonging to $\widehat{\mathcal{F}}$. 

We start by considering the case in which $C$ does not contain any elements of $\widehat{\mathcal{F}}$.
In this case, $C$ also cannot contain $\widehat{\p}$ since it only connects to $\widehat{\mathcal{F}}$. So the cycle is entirely contained in the submatrix obtained from $S_{P_{\p}}$ by removing the rows indexed by $\widehat{\mathcal{F}}$ and the column indexed by $\widehat{\p}$. We can identify this submatrix with $S_P$. Under this identification we have that $f_C$ is in $I_P$, but any $(d+2)$-minor of $S_P$ can be completed to a $(d+3)$-minor of  $S_{P_{\p}}$ by adding one row from $\widehat{\mathcal{F}}$ and column $\widehat{\p}$. The block structure of such a matrix guarantees that we are only multiplying the old minor by a new variable that will be saturated out, hence $f_C \in I_P \subseteq I_{P_{\p}}$ and we obtain the desired result.

Suppose now that $C$ includes an element $\widehat{F}$ in $\widehat{\mathcal{F}}$. We consider the following two cases.

If $\widehat{F}$ is not connected to $\widehat{\p}$ in $C$, then we can write
\[ C = \vv \widehat{F} \w \Gamma \vv \]
where $\vv, \w \in V$ and $\Gamma$ is a path.
Then we can write $C = C_1 + C_2$ where
\[ C_1 = \vv \overline{F} \w \Gamma \vv, \, C_2 = \overline{F} \vv \widehat{F} \w \overline{F}.\]
Now $C_1$ is an oriented cycle containing one fewer element of $\widehat{\mathcal{F}}$ than $C$, so $f_{C_1}\in I_{P_{\p}}$ by induction. Also $C_2$ is a cycle of type $A$, so $f_{C_2}\in I_{P_{\p}}$ by assumption. Thus by the Cycle-Splitting Lemma, we conclude that $f_C \in I_{P_{\p}}$. 

On the other hand, if $\widehat{F}$ is connected to $\widehat{\p}$ in $C$, then we can write
\[ C = \vv \widehat{F} \widehat{\p} \widehat{G} \w \Gamma \vv \]
where $\vv, \w \in V$, $\widehat{G} \in  \widehat{\mathcal{F}}$ and $\Gamma$ is a path. In this case, we can write $C = C_1 + C_2$ where
\[ C_1 = \vv \overline{F} \overline{\p} \overline{G} \w \Gamma \vv, \, C_2 = \overline{\p}\overline{F} \vv \widehat{F} \widehat{\p} \widehat{G} \w \overline{G}\overline{\p}.\]
Again $C_1$ contains fewer elements of $\widehat{\mathcal{F}}$ than $C$, so by induction $ f_{C_1}\in I_{P_{\p}}$. Also $C_2$ is a cycle of type $B$, so $f_{C_2}\in I_{P_{\p}}$ by assumption.  Thus by the Cycle-Splitting Lemma, we conclude that $f_C \in I_{P}$.

\end{proof}

\begin{corollary} \label{cor:typeA} Suppose $T_P\subseteq I_P$ and the binomials associated to all cycles of type A belong to $I_{P_{\p}}$. If $C$ is an oriented cycle of $\G_{P_{\p}}$ that either does not contain $\widehat{\p}$ or does not contain $\overline{\p}$, then $f_C \in I_{P_{\p}}$.  
\end{corollary}

\begin{proof} If $C$ does not contain $\widehat{\p}$, then the conclusion follows as in the proof of Proposition \ref{prop:vertsplitcycleclasses} because cycles of type B are only invoked when $\widehat{\p}$ appears in the original cycle. If $C$ does not contain $\overline{\p}$, then again apply the proof of Proposition \ref{prop:vertsplitcycleclasses} but interchanging $\widehat{\mathcal{F}}$ with $\overline{\mathcal{F}}$ and $\widehat{\p}$ with $\overline{\p}$. 
\end{proof}

We proceed to show when even these special cycles need not be checked. This can be characterized purely in terms of the non-incidence graph of the original polytope as follows. 

\begin{lemma} \label{lem:cycleaux}
Let $P$ be a polytope such that $T_P \subseteq I_P$ and $\G$ be the connected component of $\G_P$ that contains $\p$.  Then:
\begin{enumerate}
\item If the graph obtained from $\G$ by removing $\p$ is connected, then $T_{P_{\p}}\subseteq I_{P_{\p}}$ if and only if the binomials associated to all cycles of type A are in $I_{P_{\p}}$.
\item If the graph obtained from $\G$ by removing $\p$ and all its neighbors is connected, then the binomials associated to all cycles of types A are in $I_{P_{\p}}$.
\end{enumerate}
\end{lemma}

\begin{proof}
  We will start by proving (1). Let $C$ be a cycle of type B of the form $\overline{\p} \overline{F} \vv \widehat{F} \widehat{\p} \widehat{G} \w \overline{G} \overline{\p}$. Since $\p$ is not a cut-vertex of $\G$, there is a path $\Gamma$ in $\G_P$ from $\vv$ to $\w$ that does not pass through $\p$. Then $\Gamma$ is also a path in $\G_{P_{\p}}$ that does not pass through $\widehat{\mathcal{F}}$ nor $\overline{\p}$.  Let $C_1$ be the cycle in $\G_{P_{\p}}$ obtained by joining $\Gamma$ to the path $\overline{G} \overline{\p} \overline{F} \vv$ and $C_2$ be the cycle $\vv \widehat{F} \widehat{\p} \widehat{G} \w \Gamma \vv$. Now $C_1$ contains no elements of  $\widehat{\mathcal{F}}$, so $f_{C_1}\in I_{P_{\p}}$ by the same argument as in the base case of the proof of Proposition \ref{prop:vertsplitcycleclasses}.  Also, $C_2$ does not contain $\overline{\p}$ so $f_{C_2} \in  I_{P_{\p}}$  by Corollary \ref{cor:typeA}.  Since $C=C_1+C_2$, $f_C\in  I_{P_{\p}}$ by the Cycle-Splitting Lemma.  

We now prove (2). Suppose we have a cycle of type $A$ of the form  $C=\overline{F} \vv \widehat{F} \w \overline{F} $. By hypothesis the vertices $\vv$ and $\w$ are connected in $\G_P$ by a path $\Gamma$ passing only through $V$ and $\mathcal{F}$. This means that $C_1= \vv \Gamma \w \overline{F} \vv$ is a cycle that does not contain elements in $\widehat{\mathcal{F}}$, hence $f_{C_1} \in I_{P_{\p}}$, and similarly for $C_2=\vv \widehat{F} \w \Gamma \vv$. Since $C=C_1+C_2$, we have $f_C\in I_{P_{\p}}$ by the Cycle-Splitting Lemma.

\end{proof}

This Lemma can be improved by noticing that the conditions for the second statement almost always imply the conditions for the first statement.

\begin{proposition}\label{prop:cycleaux2} \label{prop:connectednesscondition}
Let $P$ be a polytope such that $T_P \subseteq I_P$ and let $\G$ be the connected component of $\G_P$ that contains $\p$. If the graph obtained from $\G$ by removing $\p$ and all its neighbors is connected,  then $T_{P_{\p}}\subseteq I_{P_{\p}}$.
\end{proposition}
\begin{proof}
Without loss of generality, we can suppose $\G_P=\G$ (that is, $P$ is not the join of two polytopes). First note that since removing $\p$ and its neighbours leaves the graph connected, the only way that removing $\p$ could disconnect $\G$ is if at least one of its neighbours becomes isolated (note that two neighbours can't connect to each other since $\G$ is bipartite). In this case there is a facet that contains every vertex but $\p$, which implies that $P$ is a pyramid with $\p$ as its apex. But vertex splitting the apex of a pyramid coincides with taking another pyramid over it, which preserves the inclusion $T_P \subseteq I_P$ (it is a special case of the join), so we are done with this case. 

On the other hand, if removing $\p$ leaves a connected graph, than we have the result by combining the two parts of Lemma \ref{lem:cycleaux}.
\end{proof}

It remains to see which polytopes do not verify the conditions of the previous Proposition. Suppose $P$ is a $d$-dimensional polytope with $N$ vertices and $\p$ one of its vertices. If removing $\p$ and its neighbors increases the number of connected components of $\G_P$, then the slack matrix of $P$ can be written as
\[
S_P=\kbordermatrix{
& V_1 & \p & V_2\\
\mathcal{F}_1 & A & \mathbf{0} & O\\
\mathcal{F}_2 & O & \mathbf{0} & B\\
\overline{\mathcal{F}} & \overline{A}& \mathbf{1}  & \overline{B}}.\ \]
By Corollary 3 of \cite{GGKPRT13}, a matrix is the slack matrix of a polytope if and only if its rows generate the cone obtained by intersecting its row space with the nonnegative orthant.
In this case we get a simple characterization of the extreme rays of $\Row(S_P) \cap \RR_+^N$ in terms of the matrices
\[
S'=\kbordermatrix{
& V_1 & \p\\
\mathcal{F}_1 & A & \mathbf{0}\\
\overline{\mathcal{F}} & \overline{A}&  \mathbf{1}
} \ \ \text{ and }
S''=\kbordermatrix{
& \p & V_2\\
\mathcal{F}_2 & \mathbf{0} & B\\
\overline{\mathcal{F}} & \mathbf{1} & \overline{B}  
} . \] 
\begin{lemma} Let $n := |V_1|$ and $m := |V_2|$. A vector $(p,r,q) \in \RR^{n+1+m}$ is a generator of $\Row(S_P) \cap \RR_+^{n+1+m}$ if $(p,r)$ and $(r,q)$ are generators of $\Row(S') \cap \RR_+^{n+1}$ and $\Row(S'') \cap \RR_+^{1+m}$ respectively.
\end{lemma}
\begin{proof}
Since $\dim \text{Row}(S_P) = \dim \text{Col}(S_P)$, $\text{rank} [\overline{A}\ \mathbf{1}\ \overline{B}] = 1$. Thus $\text{Row}(S_P)=(\text{Row}(A),\mathbf{0},O)+\RR(a,1,b)+(O,\mathbf{0},\text{Row}(B))\subseteq \RR^{n+1+m}$, where $[a\ 1\ b]$ is the first row indexed by $\overline{\mathcal{F}}$ in $S_P$. Note that the projection of $\text{Row}(S_P) \cap \RR_{+}^{n+1+m}$ into the first $n+1$ coordinates is simply $\text{Row}(S') \cap \RR_+^{n+1}$, while the projection into the last $1+m$ is $\text{Row}(S'') \cap \RR_+^{1+m}$. The inverse image of an extreme ray by a linear projection is a face, and the intersection of faces is still a face. Since the intersection of the inverse images, by each of the projections, of the rays generated by $(p,r)$ and $(r,q)$ is the ray generated by $(p,r,q)$, then it must be a face, hence extreme, whenever $(p,r)$ and $(r,q)$ are.
\end{proof}

This means that every generator of the cone $\row(S') \cap \RR^{n+1}_+$ appears as a row in $S'$, and similarly for $S''$. Then by removing redundant and repeated rows from each of $S'$ and $S''$, we obtain the slack matrices of some polytopes $Q$ and $R$. Now the slack matrix of $Q\oplus_{\p} R$ is a submatrix of $S_P$ whose row space equals the row space of $S_P$. But $S_P$, being the slack matrix of $P$, has no redundant rows. So in fact this submatrix is all of $S_P$. That is, $P=Q\oplus_{\p} R$.  

This allows us to extract a purely geometrical sufficient condition for vertex splitting to preserve graphicality, and it once more matches McMullen's condition.
\begin{theorem}
If $P$ satisfies $T_P \subseteq I_P$ and is not the vertex sum of two polytopes at $\p$, then the vertex splitting of $P$ at $\p$ verifies $T_{P_{\p}}\subseteq I_{P_{\p}}$. In particular, if $P$ is graphic and not the vertex sum of two polytopes at $\p$, then $P_{\p}$ is graphic.
\end{theorem}

  \section{Operations on Finite Posets and their Order Polytopes} \label{sec:posets}

In this section we turn our attention to order polytopes of posets. It turns out that several simple operations on posets correspond precisely to applying the operations introduced in Section \ref{sec:orderpoly} to their order polytopes. This opens the possibility of applying the results developed in the previous section to the question of graphicality of order polytopes, a task that we will undertake in Section \ref{sec:conclusions}. 

We start by defining three simple operations that depend only on the posets involved, without any additional choices.

\begin{definition}   \label{def:orderpolymain}
Let $\Pscr$ and $\Qscr$ be posets on the disjoint sets $X$ and $Y$ respectively. We then define the following new posets.
\begin{enumerate}
\item $\Pscr_{\rev}$, the \emph{reverse} of $\Pscr$, is the poset on $X$ obtained by reversing the order of $\Pscr$.
\item $\Pscr\vee \Qscr$, the \emph{join} of $\Pscr$ and $\Qscr$, is the poset defined by taking all the elements and relations of $\Pscr$ and $\Qscr$, and adding a new element, $\ast$, that is greater than each element in $X$ and less than each element in $Y$.
\item $\Pscr\oplus \Qscr$, the \emph{ordinal sum} of $\Pscr$ and $\Qscr$, is the poset defined by taking all the elements and relations of $\Pscr$ and $\Qscr$ and imposing that each element of $X$ is less than every element of $Y$.
\item $\Pscr + \Qscr$, the \emph{direct sum} of $\Pscr$ and $\Qscr$, is the poset defined by simply taking the union of all the elements and relations of $\Pscr$ and $\Qscr$, with no additional relations.
\end{enumerate}
\end{definition}

\begin{example}\label{ex:posetop1}
Consider the posets $\Pscr$ and $\Qscr$ given by the following Hasse diagrams.
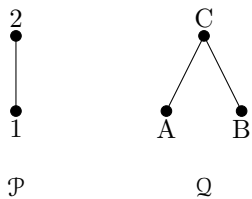
\[\begin{tikzpicture}
\filldraw (0,0) node[above] {2} circle (2pt);
\filldraw (0,-1) node[below] {1} circle (2pt);
\draw (-0,-1) -- (0,0);
\draw (0, -2) node {$\Pscr$};
\filldraw (2.5,0) node[above] {C} circle (2pt);
\filldraw (2,-1) node[below] {A} circle (2pt);
\filldraw (3, -1) node[below] {B} circle (2pt);
\draw (2,-1) -- (2.5,0) -- (3, -1);
\draw (2.5, -2) node {$\Qscr$};
\end{tikzpicture}\]
We illustrate the operations defined above applied to $\Pscr$ and $\Qscr$. 
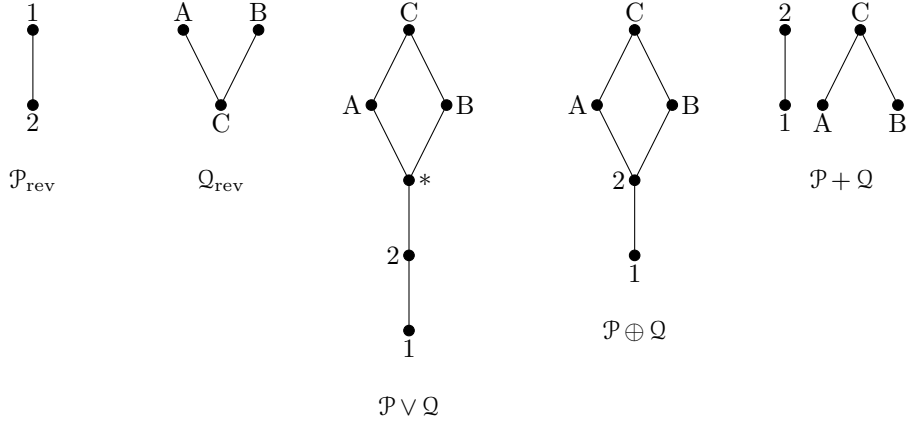
\[\begin{tikzpicture}
\filldraw (0,0) node[above] {1} circle (2pt);
\filldraw (0,-1) node[below] {2} circle (2pt);
\draw (-0,-1) -- (0,0);
\draw (0, -2) node {$\Pscr_{\rev}$};
\filldraw (2.5,-1) node[below] {C} circle (2pt);
\filldraw (2,0) node[above] {A} circle (2pt);
\filldraw (3, 0) node[above] {B} circle (2pt);
\draw (2,0) -- (2.5,-1) -- (3,0);
\draw (2.5, -2) node {$\Qscr_{\rev}$};
\begin{scope}[xshift=5cm]
\filldraw (0,0) node[above] {C} circle (2pt);
\filldraw (-0.5,-1) node[left] {A} circle (2pt);
\filldraw (0.5, -1) node[right] {B} circle (2pt);
\draw (-0.5,-1) -- (0,0) -- (0.5, -1);
\filldraw (0,-3) node[left] {2} circle (2pt);
\filldraw (0,-4) node[below] {1} circle (2pt);
\draw (0,-3) -- (0,-4);
\filldraw (0,-2) node[right] {$\ast$} circle (2pt);
\draw (-0.5,-1) -- (0,-2) -- (0.5, -1);
\draw (0,-3) -- (0,-2);
\draw (0, -5) node {$\Pscr\vee \Qscr$};
\end{scope}
\begin{scope}[xshift=8cm]
\filldraw (0,0) node[above] {C} circle (2pt);
\filldraw (-0.5,-1) node[left] {A} circle (2pt);
\filldraw (0.5, -1) node[right] {B} circle (2pt);
\draw (-0.5,-1) -- (0,0) -- (0.5, -1);
\filldraw (0,-2) node[left] {2} circle (2pt);
\filldraw (0,-3) node[below] {1} circle (2pt);
\draw (0,-2) -- (0,-3);
\draw (-0.5,-1) -- (0,-2) -- (0.5, -1);
\draw (0, -4) node {$\Pscr\oplus \Qscr$};
\end{scope}
\begin{scope}[xshift=10cm]
\filldraw (0,0) node[above] {2} circle (2pt);
\filldraw (0,-1) node[below] {1} circle (2pt);
\draw (-0,-1) -- (0,0);
\draw (0.75, -2) node {$\Pscr + \Qscr$};
\filldraw (1,0) node[above] {C} circle (2pt);
\filldraw (0.5,-1) node[below] {A} circle (2pt);
\filldraw (1.5, -1) node[below] {B} circle (2pt);
\draw (0.5,-1) -- (1,0) -- (1.5,-1);
\end{scope}
\end{tikzpicture}\]
\end{example}

We will now see that all these operations on posets induce simple operations on their order polytopes.

\begin{proposition} \label{prop:orderpolymain}
Let $\Pscr$ and $\Qscr$ be finite posets on disjoint ground sets $X$ and $Y$. Then we have the following relations between order polytopes.
\begin{enumerate}
\item $\Ord(\Pscr)$ is affinely equivalent to $\Ord(\mathscr{P_{\rev}})$;
\item $\Ord(\Pscr\vee\Qscr)$ is combinatorially equivalent to $\Ord(\Pscr)\vee\Ord(\Qscr)$;
\item  $\Ord(\Pscr\oplus\Qscr)$ is combinatorially equivalent to ${\Ord(\Pscr)}\oplus_{(\vv,\w)}{\Ord(\Qscr)}$, where $\vv$ is the vertex in $\Ord(\Pscr)$ given by the empty filter and $\w$ is the vertex in $\Ord(\Qscr)$ given by the complete filter. 
\item $\Ord(\Pscr+\Qscr)=\Ord(\Pscr) \times \Ord(\Qscr)$.
\end{enumerate}
\end{proposition}
\begin{proof}

(1)  The affine function $\phi:\mathbb{R}^d\rightarrow\mathbb{R}^d$ given by  $\phi(\x)=1-\x$ is an affine isomorphism between $\Ord(\Pscr)$ and $\Ord(\mathscr{P_{\rev}})$.

(2) Note that we can naturally identify the facets of $\Ord(\Pscr)$ and $\Ord(\Qscr)$ with those of $\Ord(\Pscr\vee\Qscr)$. Facets coming from cover relations, minimal elements of $\Pscr$ and maximal elements of $\Qscr$ are still present in $\Ord(\Pscr\vee\Qscr)$. Those coming from maximal elements of $\Pscr$ can be identified with inequalities given by the cover relations of $\ast$, and similarly with the minimal elements of $\Qscr$.

As for the vertices, note that if any filter of $\Pscr\vee\Qscr$ contains an element of $X$, then it also contains $\ast$ and every element of $Y$. So any filter either does not contain $\ast$ nor any element of $X$, or it contains $\ast$ and every element of $Y$. We can identify the filters of $\Pscr$ with those of $\Pscr\vee\Qscr$ that contain $\ast$ by adding $\ast$ and all elements of $Y$, while the filters of $\Qscr$ identify directly with filters of $\Pscr\vee\Qscr$.

We thus have an identification between the facets of $\Ord(\Pscr\vee\Qscr)$ and the union of those of $\Ord(\Pscr)$ and $\Ord(\Qscr)$ and similarly for vertices. It is now easy to check that under that identification, every vertex of  $\Ord(\Qscr)$ belongs to every facet of $\Ord(\Pscr)$ and vice-versa. Moreover, the identifications respect vertex-facets incidences inside each of the posets. This implies that $\Ord(\Pscr\vee\Qscr)$ and $\Ord(\Pscr)\vee\Ord(\Qscr)$ have the same vertex-facet incidences, hence are combinatorially equivalent.

(3)
We will proceed as in (2). Again, note that facets coming from cover relations, minimal elements of $\Pscr$ and maximal elements of $\Qscr$ are still present in $\Ord(\Pscr\oplus\Qscr)$. Furthermore, one can map filters of $\Pscr$ to filters of $\Pscr\oplus\Qscr$ by adding all elements of $Y$, while filters of $\Qscr$ again directly yield filters of $\Pscr\oplus\Qscr$. However this will map two filters to the same: the empty filter of $\Pscr$ and the complete filter $Y$ of $\Qscr$. Moreover, facets  of $\Ord(\Pscr)$ and $\Ord(\Qscr)$ that do not contain these filters are precisely those that correspond to maximal elements of  $\Pscr$ and minimal elements of $\Qscr$, respectively, and are not present in $\Ord(\Pscr\oplus\Qscr)$. Instead for any pair of such facets, we have a single new facet given by the cover relation between the maximal element of $\Pscr$ and the minimal element of $\Qscr$ that was introduced by $\oplus$. This gives us again a one to one identification between facets and vertices of $\Ord(\Pscr\oplus\Qscr)$ and those of ${\Ord(\Pscr)}\oplus_{(\vv,\w)}{\Ord(\Qscr)}$, and it is once again easy to see that the incidence relations are preserved. Thus $\Ord(\Pscr\oplus\Qscr)$ and ${\Ord(\Pscr)}\oplus_{(\vv,\w)}{\Ord(\Qscr)}$ are combinatorially equivalent.

(4) We need only to observe that a set of elements of $X \cup Y$ is a filter on $\Pscr+\Qscr$ if and only if its restrictions to $X$ and $Y$ are filters of $\Pscr$ and $\Qscr$ respectively. This means that the set of vertices of $\Ord(\Pscr\vee\Qscr)$ is the set of all pairs $(\vv,\w)$ where $\vv$ and $\w$ are vertices of $\Ord(\Pscr)$ and $\Ord(\Qscr)$ respectively, proving the result.
\end{proof}

As a particular case of (2), note that given a poset $\Pscr$, the poset $\Pscr^{\Delta}$ obtained from $\Pscr$ by adjoining a new universal maximum is simply the join of $\Pscr$ with the empty poset, which implies that $\Ord(\Pscr^{\Delta})$ is simply the join of $\Ord(\Pscr)$ with a point, i.e., a pyramid over $\Ord(\Pscr)$.

We now introduce two other operations on posets that depend on more than just their ground sets.

\begin{definition}
  Let $\Pscr$ be a poset.
\begin{enumerate}
\item For  $a\cover b$ in $\Pscr$ we define  $\Pscr_{a\cover b}$ by adding a new element $\ast$ and replacing $a\cover b$ by $a \cover \ast \cover b$. We say that
$\Pscr_{a\cover b}$ is obtained from $\Pscr$ by \emph{splitting the cover} $a\cover b$.
\item For $a$ a maximal (minimal) element of  $\Pscr$ we define  $\Pscr_{a}$ by adding a new element $\ast$ and the cover $a \cover \ast$ (respectively $\ast \cover a$).
We say that
$\Pscr_{a}$ is obtained from $\Pscr$ by \emph{splitting the maximal (minimal) element} $a$.
\end{enumerate}
\end{definition}

\begin{example}
Let $Q$ be the poset of Example \ref{ex:posetop1}. If we split the cover $A\cover C$ or the maximal element $C$ we obtain the following posets. 
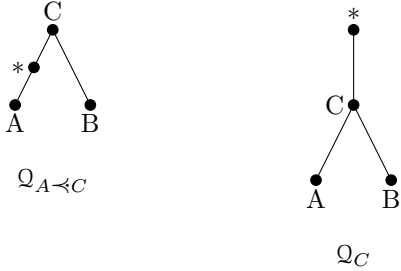
\[\begin{tikzpicture}
\begin{scope}[xshift=-2cm]
\filldraw (2.5,0) node[above] {C} circle (2pt);
\filldraw (2.25,-0.5) node[left] {$\ast$} circle (2pt);
\filldraw (2,-1) node[below] {A} circle (2pt);
\filldraw (3, -1) node[below] {B} circle (2pt);
\draw (2,-1) -- (2.5,0) -- (3, -1);
\draw (2.5, -2) node {$\Qscr_{A\cover C}$};
\end{scope}
\begin{scope}[xshift=2cm]
\filldraw (2.5,0) node[above] {$\ast$} circle (2pt);
\filldraw (2.5,-1) node[left] {C} circle (2pt);
\filldraw (2,-2) node[below] {A} circle (2pt);
\filldraw (3, -2) node[below] {B} circle (2pt);
\draw (2,-2) -- (2.5,-1) -- (3, -2);
\draw (2.5,-1) -- (2.5,0);
\draw (2.5, -3) node {$\Qscr_C$};
\end{scope}
\end{tikzpicture}\]
\end{example}

Since covers and maximal/minimal elements determine the facets of the order polytope, it is not surprising that the operations of splitting in finite posets are related to the operation of facet wedging on polytopes.

\begin{proposition}\label{prop:splitting}
Let $\Pscr$ be a poset, let $a\cover b$ be a cover relation in $\Pscr$ and $c$ be a minimal or maximal element. Then
\begin{enumerate}
\item $\Ord(\Pscr_{a\cover b})$ is the facet wedge of $\Ord(\Pscr)$ at $F:t_a \leq t_b$;
\item $\Ord(\Pscr_{c})$ is the facet wedge of $\Ord(\Pscr)$ with respect to the facet cut by $0 \leq t_c$ or $t_c \leq 1$, depending on whether $c$ is minimal or maximal.
\end{enumerate}
 \end{proposition}
\begin{proof}
We will prove only (1), since the proof of (2) is completely analogous. Note that splitting a cover relation replaces one facet given by $F:t_a \leq t_b$ by two new ones given by $\hat{F}:t_a \leq t_\ast$ and $\tilde{F}:t_\ast \leq t_b$. In terms of vertices,  if $b$ is not in a filter, then neither is $\ast$, while if $a$ is in the filter then so is $\ast$. This implies that there is a bijection between vertices of
$\Ord(\Pscr)$ that are in the facet given by $t_a\leq t_b$ and vertices $\vv$ in $\Ord(\Pscr_{a\cover b})$ that satisfy $v_a = v_\ast =v_b$. However, if $a$ is not in the filter and $b$ is, then adding $\ast$ to the filter maintains it as a filter. So each vertex $\vv$ in $\Ord(\Pscr)$ with $v_a=0$ and $v_b=1$ corresponds to two vertices in $\Ord(\Pscr_{a\cover b})$: the first vertex $\hat{\vv}$ is obtained by adding $v_\ast=0$ and the second one $\tilde{\vv}$ is obtained by adding $v_\ast=1$. If the slack matrix of $\Ord(\Pscr)$ is of the form below on the left, then that of $\Ord(\Pscr_{a\cover b})$ will be as below on the right.
\[\kbordermatrix{
& V &  W\\
 & A & B \\
F & \mathbf{1}^T & \mathbf{0}^T
}, \hspace{3cm} \kbordermatrix{
& \hat{V} & \tilde{V} & W\\
 & A & A & B \\
\hat{F}  & \mathbf{0}^T & \mathbf{1}^T & \mathbf{0}^T\\
\tilde{F} & \mathbf{1}^T & \mathbf{0}^T & \mathbf{0}^T
}.\]
The second matrix is precisely the slack matrix of the facet wedge of $\Ord(\Pscr)$ at $F$, giving us the desired result.
\end{proof}

The last operation we will introduce is a weaker version of the ordinal sum.

\begin{definition}  \label{def:orderpolyn}
Let  $\Pscr$ and $\Qscr$ be posets on disjoint sets, $a$ a maximal element of $\Pscr$ and $b$ a minimal element of $\Qscr$. The \emph{the partial ordinal sum of $\Pscr$ and $\Qscr$ with respect to $a$ and $b$}, denoted by $\Pscr \oplus_{(a,b)} \Qscr$ is the poset attained by taking all elements and relations of $\Pscr$ and $\Qscr$ and adding the relations that $a$ is less than all elements of  $\Qscr$ while $b$ is greater than all elements of $\Pscr$.
\end{definition}

\begin{example}
Consider the posets $\Rscr$ and $\Qscr$ with the Hasse diagrams below on the left and center. On the right is the Hasse diagram of the partial ordinal sum of $\Rscr$ and $\Qscr$ with respect to $4$ and $A$.
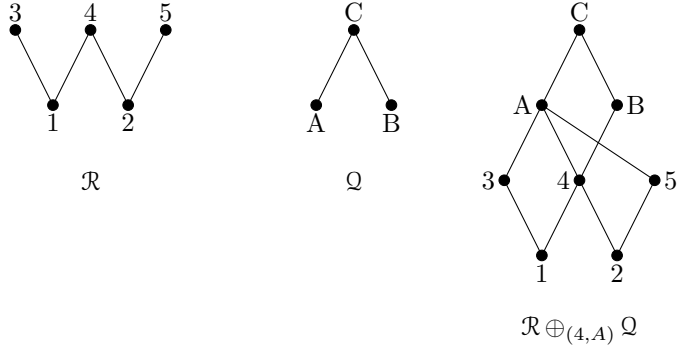
\[\begin{tikzpicture}
\filldraw (0,0) node[above] {3} circle (2pt);
\filldraw (1,0) node[above] {4} circle (2pt);
\filldraw (2,0) node[above] {5} circle (2pt);
\filldraw (0.5,-1) node[below] {1} circle (2pt);
\filldraw (1.5,-1) node[below] {2} circle (2pt);
\draw (0,0) -- (0.5,-1) -- (1,0) -- (1.5,-1) -- (2,0);
\draw (1,-2) node {$\Rscr$};
\begin{scope}[xshift=2cm]
\filldraw (2.5,0) node[above] {C} circle (2pt);
\filldraw (2,-1) node[below] {A} circle (2pt);
\filldraw (3, -1) node[below] {B} circle (2pt);
\draw (2,-1) -- (2.5,0) -- (3, -1);
\draw (2.5, -2) node {$\Qscr$};
\end{scope}
\begin{scope}[xshift=5cm]
\filldraw (2.5,0) node[above] {C} circle (2pt);
\filldraw (2,-1) node[left] {A} circle (2pt);
\filldraw (3, -1) node[right] {B} circle (2pt);
\draw (2,-1) -- (2.5,0) -- (3, -1);
\filldraw (1.5,-2) node[left] {3} circle (2pt);
\filldraw (2.5,-2) node[left] {4} circle (2pt);
\filldraw (3.5,-2) node[right] {5} circle (2pt);
\filldraw (2,-3) node[below] {1} circle (2pt);
\filldraw (3,-3) node[below] {2} circle (2pt);
\draw (1.5,-2) -- (2,-3) -- (2.5,-2) -- (3,-3) -- (3.5,-2);
\draw (1.5,-2) -- (2,-1) -- (3.5,-2);
\draw (2,-1) -- (2.5,-2) -- (3, -1);
\draw (2.5, -4) node {$\Rscr \oplus_{(4,A)} \Qscr$};
\end{scope}
\end{tikzpicture}\]

\end{example}

Once more this operation has a simple interpretation in terms of its effect on the order polytopes.

\begin{proposition} \label{prop:orderpolyn}
Let  $\Pscr$ and $\Qscr$ be finite posets on disjoint sets, $a$ a maximal element of $\Pscr$ and $b$ a minimal element of $\Qscr$. Then $\Ord(\Pscr \oplus_{(a,b)} \Qscr)$ is the facet product $\Ord(\Pscr) \otimes_F \Ord(\Qscr)$ where we identify
with $F$ the facets $F_1:t_a \leq 1$ of $\Ord(\Pscr)$ and $F_2:t_b \geq 0$ of $\Ord(\Qscr)$.
\end{proposition}
\begin{proof}
In terms of facets, all the facets except the ones corresponding to maximal elements of $\Pscr$ and minimal elements of $\Qscr$ are preserved untouched. For the maximal elements $c$ of $\Pscr$ different from $a$,
we can identify the facets of $\Ord(\Pscr)$ given by $t_c \leq 1$ with the new facets $t_c \leq t_b$ of $\Ord(\Pscr \oplus_{(a,b)} \Qscr)$. We can similarly identify the facets corresponding to minimal elements $d$
of $\Qscr$ different from $b$ with the facets $t_a \leq t_d$. We are left with facets $F_1$ and $F_2$, corresponding to the two special extremal elements, that will disappear and be replaced with a single facet $F:t_a \leq t_b$.

In terms of vertices, any filter of $\Pscr$ that contains $a$ gives rise to a filter of $\Pscr \oplus_{(a,b)} \Qscr$ only by adding all elements of $\Qscr$, and those are the only filters in that poset that contain $a$. Similarly, any filter in $\Qscr$ that does not contain $b$ is also a filter of $\Pscr \oplus_{(a,b)} \Qscr$, and those are the only filters in that polytope that do not contain $b$. This means that the only vertices that we have to deal with are those filters of $\Pscr \oplus_{(a,b)} \Qscr$ that do not contain $a$ but do contain $b$. The restriction of any such filter to the elements of $\Pscr$ and $\Qscr$ gives rise to a pair of filters in those posets. Moreover, the union of a filter of $\Pscr$ not containing $a$ and a filter of $\Qscr$ containing $b$ is always a valid filter in $\Pscr \oplus_{(a,b)} \Qscr$.

This means that a vertex in $\Ord(\Pscr \oplus_{(a,b)} \Qscr)$ can be identified with either a vertex in $V$, where $V$ is the set of vertices in $\Ord(\Pscr)$  in the facet $F_1$,  with a vertex in $W$, where $W$ is the set of vertices in $\Ord(\Qscr)$ in the facet $F_2$, or with a pair of vertices $(v,w) \in \bar{V} \times \bar{W}$ where $\bar{V}$ and $\bar{W}$ are, respectively, the vertices of $\Ord(\Pscr)$ not in $F_1$ and those of of $\Ord(\Qscr)$ not in $F_2$. It is now easy to check that if the slack matrices of  $\Ord(\Pscr)$ and  $\Ord(\Qscr)$ are the ones below on the left and center, with the identifications introduced above, then the slack matrix of $\Ord(\Pscr \oplus_{(a,b)} \Qscr)$ is the one below on the right.
\[\kbordermatrix{
    & V &  \bar{V}\\
\mathcal{F}  & A & B \\
F_1 & \mathbf{0}^T & \mathbf{1}^T
} \hspace{1cm}
\kbordermatrix{
    & W &  \bar{W}\\
\mathcal{F}'    & C & D \\
F_2 & \mathbf{0}^T & \mathbf{1}^T
} \hspace{1cm}
\kbordermatrix{
              & V & W & \bar{V} \times \bar{W} \\
 \mathcal{F}  & A & O & B \otimes \mathbf{1}^T \\
 \mathcal{F}' & O & C & \mathbf{1}^T \otimes D \\
         {F}  & \mathbf{0}^T & \mathbf{0}^T & \mathbf{1}^T
}\]
The last matrix is precisely the slack matrix of $\Ord(\Pscr) \otimes_F \Ord(\Qscr)$, concluding the proof.
\end{proof}

  \section{Graphicality of order polytopes} \label{sec:conclusions}

  In this section, we put together the work developed in the previous two sections to derive sufficient conditions for graphicality, and consequently projective uniqueness, to arise in order polytopes. In particular, we will prove that every finite ranked poset with no 3-antichain has a graphic order polytope.

  We begin by stating several graphicality results that follow immediately from the descriptions of order polytopes in Section 4 along with the general results on graphicality in Section 3. The operations denoted in this statement are the ones introduced in Definitions \ref{def:orderpolymain} and \ref{def:orderpolyn}.

  \begin{proposition} \label{cor:simplePU} Let $\Pscr$ and $\Qscr$ be finite posets on disjoint ground sets. Then if $\Ord(\Pscr)$ and $\Ord(\Qscr)$ are projectively unique (respectively graphic),
  so are $\Ord(\Pscr_{\rev})$, $\Ord(\Pscr \vee \Qscr)$, $\Ord(\Pscr \oplus \Qscr)$ and $\Ord(\Pscr \oplus_{(a,b)} \Qscr)$.
  \end{proposition}
\begin{proof}
  This is immediate from Propositions \ref{prop:orderpolymain} and \ref{prop:orderpolyn} that describe the effect of the operations in the order polytopes and Theorems \ref{McOperations} and \ref{thm:main} that show that they all preserve projective uniqueness and graphicality.
\end{proof}

Another operation that preserves graphicality is the splitting of extremal elements, but that is a little more delicate to show.

\begin{proposition}\label{sm}
Let $\Pscr$ be a poset such that $\Ord(\Pscr)$ is graphic, and $c$ one of its extremal elements. Then
$\Ord(\Pscr_c)$  is graphic.
\end{proposition}
\begin{proof}
From Proposition \ref{prop:splitting} we know that the operation of splitting $c$ correspond to facet wedging in the order polytope with respect to the facet $F$ given by the extremal element $c$.
Moreover, Proposition \ref{prop:connectednesscondition} gives us a sufficient condition for such operation to preserve graphicality: that removing the facet $F$ and all of its neighbors from the non-incidence graph of $\Ord(\Pscr)$ does not create any new connected components. 

We may assume $c$ is a maximal element since reversing the poset preserves graphicality. Take any pair of facets $\bar{F}$ and $\tilde{F}$. There must be a path $\Gamma$ between them in the original non-incidence graph, say $\bar{F} w_0 F_1 w_1 \dots F_t w_t \tilde{F}$. Recall that each $w_i$ corresponds to a filter of the poset and that adding a maximal element to a filter preserves the filter property, so let $w'_i =w_i \cup\{c\}$. None of the $w'_i$ is a neighbor of $F$.

Now form a new sequence $\Gamma'$ by starting with $\Gamma$ and replacing $w_i$ by $w_i'$ for each $i$. If $w'_i \neq w_i$, then the only facet that neighbors $w_i$ but not  $w_i'$ is $F$, so the only potential problem is that $F$ might belong to $\Gamma$, in which case $\Gamma'$ lacks one edge to be a path. In this case, let $F'$ be the facet associated to a cover $d \cover c$ (or the facet induced by $c$ being a minimal element if $c$ is both maximal and minimal) and modify $\Gamma'$ by replacing $F$ by $F'$. Now if $w_i$ neighbors $F$, then it corresponds to a filter that does not contain $c$, hence also does not contain $d$. Then $w_i'$ corresponds to a filter that does contain $c$ but does not contain $d$, so $w_i'$ neighbors $F'$. Thus $\Gamma'$ is now a path from $\bar{F}$ to $\tilde{F}$ that avoides $F$ and all of its neighbors, so the condition in Proposition \ref{prop:connectednesscondition} is satisfied.\end{proof}

Two other operations were defined in Section 4: the direct sum of posets and cover splitting. It is not hard to see that these operations do not universally preserve graphicality.
\begin{example}
If we consider $\Pscr$ and $\Qscr$ to be posets on one and two elements and no relations as represented below, then their order polytopes are a segment and a square, both of which are graphic. However
$\Ord(\Pscr + \Qscr)$ is a cube, which is not projectively unique and hence not graphic

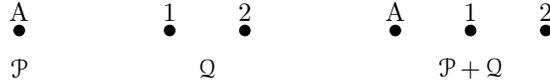
\[\begin{tikzpicture}
\filldraw (2,0) node[above] {1} circle (2pt);
\filldraw (3,0) node[above] {2} circle (2pt);
\draw (2.5,-0.5) node {$\Qscr$};

\filldraw (0,0) node[above] {A} circle (2pt);
\draw (0, -0.5) node {$\Pscr$};

\begin{scope}[xshift=4cm]
\filldraw (2,0) node[above] {1} circle (2pt);
\filldraw (3,0) node[above] {2} circle (2pt);
\filldraw (1,0) node[above] {A} circle (2pt);
\draw (2, -0.5) node {$\Pscr+\Qscr$};
\end{scope}
\end{tikzpicture}\]
Similarly, if $\Rscr$ is the poset shown below, it is not hard to see that its order polytope is graphic, as $\Pscr$ is a partial ordinal sum of two posets with two elements and no relations.
However, one can computationally check that $\Rscr_{2\cover 3}$ is not projectively unique or graphic.
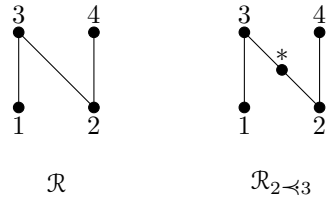
\[\begin{tikzpicture}
\filldraw (0,0) node[above] {3} circle (2pt);
\filldraw (1,0) node[above] {4} circle (2pt);

\filldraw (0,-1) node[below] {1} circle (2pt);
\filldraw (1,-1) node[below] {2} circle (2pt);
\draw (0,-1) -- (0,0) -- (1,-1) -- (1,0);
\draw (0.5,-2) node {$\Rscr$};

\begin{scope}[xshift=3cm]
\filldraw (0,0) node[above] {3} circle (2pt);
\filldraw (1,0) node[above] {4} circle (2pt);
\filldraw (0.5,-0.5) node[above] {$\ast$} circle (2pt);
\filldraw (0,-1) node[below] {1} circle (2pt);
\filldraw (1,-1) node[below] {2} circle (2pt);
\draw (0,-1) -- (0,0) -- (0.5,-0.5) -- (1,-1) -- (1,0);
\draw (0.5,-2) node {$\Rscr_{2\cover 3}$};
\end{scope}
\end{tikzpicture}\]
\end{example}

In both of these cases, non-graphicality seems related to the antichain of size three. If $\Pscr$ is a poset with an antichain of size three, then $\Ord(\Pscr)$ always has a face which is a $3$-cube. This face is not projectively unique. This does not imply that $\Ord(\Pscr)$ itself cannot be projectively unique (see the discussion of non-prescribable faces in \cite{GMTW18-1}) but it strongly suggests that it may not be. If we rule these antichains out, we can show graphicality for ranked posets.

\begin{theorem}\label{mt}
Let $\Pscr$ be a finite ranked poset with no 3-antichain. Then $\Ord(\Pscr)$ is graphic, and therefore projectively unique.
\end{theorem}
\begin{proof}

  We prove this by induction on the rank of $\Pscr$. If $\Pscr$ has rank 0, then its order polytope is a segment or a square, both of which are graphic. Also note that if there are no 3-antichains, then the only possibility for the Hasse diagram to be disconnected is that  $\Pscr$ is the direct sum of two chains. In this case, $\Ord(\Pscr)$ is graphic, as it can be attained from successively splitting maximal elements starting with the poset of two unrelated elements. Thus our strategy will be to assume that this happens for all finite ranked posets of rank $n-1$ and prove it for finite ranked posets of rank $n$ such that the Hasse diagram is connected.
  
Note that elements of top rank are maximal and elements of constant rank form an antichain. Let $\Pscr_0$ be the poset obtained from $\Pscr$ by removing its top ranked elements, which we can suppose without loss of generality have rank $n$. The Hasse diagram of  $\Pscr$ is constructed from that of $\Pscr_0$ by connecting the rank $n$ elements to the rank $n-1$ elements in one of the following ways.
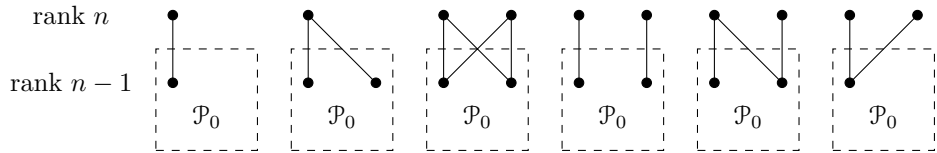
\[\begin{tikzpicture}[scale=0.9]
\filldraw (-1.5,1) node {rank $n$};
\filldraw (-1.5,0) node {rank $n-1$};
\filldraw (0,0) node {} circle (2pt);
\filldraw (0,1) node {} circle (2pt);
\filldraw (0.5,-0.5) node {$\Pscr_0$};
\draw (0,0) -- (0,1);
\draw[dashed] (-0.25,-1) rectangle (1.25, 0.5);
\begin{scope}[xshift=2cm]
\filldraw (0,0) node {} circle (2pt);
\filldraw (0,1) node {} circle (2pt);
\filldraw (1,0) node {} circle (2pt);
\filldraw (0.5,-0.5) node {$\Pscr_0$};
\draw (0,0) -- (0,1) -- (1,0);
\draw[dashed] (-0.25,-1) rectangle (1.25, 0.5);
\end{scope}
\begin{scope}[xshift=4cm]
\filldraw (0,0) node {} circle (2pt);
\filldraw (0,1) node {} circle (2pt);
\filldraw (1,0) node {} circle (2pt);
\filldraw (1,1) node {} circle (2pt);
\filldraw (0.5,-0.5) node {$\Pscr_0$};
\draw (0,0) -- (0,1) -- (1,0)--(1,1)--(0,0);
\draw[dashed] (-0.25,-1) rectangle (1.25, 0.5);
\end{scope}
\begin{scope}[xshift=6cm]
\filldraw (0,0) node {} circle (2pt);
\filldraw (0,1) node {} circle (2pt);
\filldraw (1,0) node {} circle (2pt);
\filldraw (1,1) node {} circle (2pt);
\filldraw (0.5,-0.5) node {$\Pscr_0$};
\draw (0,0) -- (0,1);
\draw (1,0) -- (1,1);
\draw[dashed] (-0.25,-1) rectangle (1.25, 0.5);
\end{scope}
\begin{scope}[xshift=8cm]
\filldraw (0,0) node {} circle (2pt);
\filldraw (0,1) node {} circle (2pt);
\filldraw (1,0) node {} circle (2pt);
\filldraw (1,1) node {} circle (2pt);
\filldraw (0.5,-0.5) node {$\Pscr_0$};
\draw (0,0) -- (0,1)--(1,0)--(1,1);
\draw[dashed] (-0.25,-1) rectangle (1.25, 0.5);
\end{scope}
\begin{scope}[xshift=10cm]
\filldraw (0,0) node {} circle (2pt);
\filldraw (0,1) node {} circle (2pt);
\filldraw (1,1) node {} circle (2pt);
\filldraw (0.5,-0.5) node {$\Pscr_0$};
\draw (1,1)--(0,0) -- (0,1);
\draw[dashed] (-0.25,-1) rectangle (1.25, 0.5);
\end{scope}
\end{tikzpicture}\]
Note now that the first and fourth cases are obtained by splitting maximal elements of $\Pscr_0$. The second, third and sixth cases are ordinal sums of  $\Pscr_0$ with the posets of a single element (in the second) and two unrelated elements (in the third and sixth). Note that, in these three cases, the rank $n-1$ elements drawn must be the full set of maximal elements of $\Pscr_0$, as otherwise there would be a $3$-antichain in $\Pscr$. Finally, the fifth case is a partial ordinal sum of  $\Pscr_0$ with the poset of two unrelated elements. Since we have seen that all these operations preserve graphicality, the result follows.
\end{proof}

Note that this Theorem is not exhaustive of all graphic order polytopes, since the ranked condition is not necessary.

\begin{example}
Let $\Pscr$ be the following poset. 
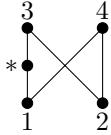
\[\begin{tikzpicture}
\filldraw (0,0) node[above] {3} circle (2pt);
\filldraw (1,0) node[above] {4} circle (2pt);
\filldraw (0,-.5) node[left] {$\ast$} circle (2pt);
\filldraw (0,-1) node[below] {1} circle (2pt);
\filldraw (1,-1) node[below] {2} circle (2pt);
\draw (1,0)--(0,-1) -- (0,0) -- (1,-1) -- (1,0);
\end{tikzpicture}\]
This is not a ranked poset, but it is easy to check computationally that it is graphic. In fact, it can be attained from a ranked poset with no $3$-antichains by splitting a cover, and one could show more generally that cover splitting preserves graphicality under mild assumptions (essentially that it does not create a $3$-antichain).
\end{example}

To conclude this discussion we present two conjectures on the necessary and sufficient conditions for order polytopes to be graphic and projectively unique. First we conjecture that the ranked condition can simply be dropped from Theorem \ref{mt}.

\begin{conjecture} \label{conj:unranked}
Let $\Pscr$ be a finite poset (not necessarily ranked) with no 3-antichain. Then $\Ord(\Pscr)$ is graphic.
\end{conjecture}

\begin{example}
  Dealing with cover-splitting would not be sufficient to prove Conjecture \ref{conj:unranked}. The following unranked poset cannot be obtained by cover-splitting or any of the previously considered operations from another poset with fewer elements.
  
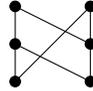
\[\begin{tikzpicture}
\filldraw (0,0) node[above] {} circle (2pt);
\filldraw (1,0) node[above] {} circle (2pt);
\filldraw (0,-.5) node[left] {} circle (2pt);
\filldraw (1,-.5) node[left] {} circle (2pt);
\filldraw (0,-1) node[below] {} circle (2pt);
\filldraw (1,-1) node[below] {} circle (2pt);
\draw (0,-1)--(0,0);
\draw (1,-1)--(1,0);
\draw (0,0)--(1,-.5);
\draw (0,-.5)--(1,-1);
\draw (0,-1)--(1,0);
\end{tikzpicture}\]
However, with the help of Antonio Macchia and Amy Wiebe and their Macaulay2 package for slack ideals \cite{MW20}, we were able to verify that the order polytope of this poset is at least projectively unique.    
  

\end{example}

Secondly, we believe that having no $3$-antichain is actually a necessary condition even for projective uniqueness.

\begin{conjecture}
Let $\Pscr$ be a finite poset. If $\Ord(\Pscr)$ is projectively unique, then $\Pscr$ has no antichain of size 3.
\end{conjecture}

Note that these two conjectures together would in particular imply that projectively unique order polytopes are all graphic.

\end{document}